\documentclass[letterpaper,11pt]{article}
\usepackage[utf8x]{inputenc}
\usepackage[title]{appendix}
\usepackage[noTeX]{mmap}
\usepackage[left=1.25in, right=1.25in, bottom = 1.25in, top = 1.25in]{geometry}
\usepackage{hyperref} 
\hypersetup{colorlinks}
\usepackage{amsmath,amsfonts, amsthm, amssymb}
\usepackage{color}
\usepackage{enumerate}
\usepackage{hypbmsec}
\usepackage{bbm}
\usepackage{dsfont}
\usepackage{mathtools}
\usepackage{graphicx}
\usepackage{caption}
\usepackage{subcaption}
\usepackage[section]{placeins}
\usepackage{tikz}
\usetikzlibrary{calc}
\usetikzlibrary{decorations.pathreplacing}
\usetikzlibrary{arrows}
\usepackage{pgffor}
\usepackage{longtable}
\usepackage{comment}
\usepackage{multirow}
\usepackage{rotating}
\usepackage{bm}
\usepackage{mdframed}
\newcommand\numberthis{\addtocounter{equation}{1}\tag{\theequation}}

\usepackage{fancyhdr}
\usepackage{tocloft}

\usepackage{mathtools,tabstackengine}
\TABstackMath% STACK MATRIX IN MATH MODE
\setstacktabbedgap{10pt}% INTER-COLUMN GAP SIZE
\setstackgap{L}{1.5\baselineskip}% INTER-ROW BASELINESKIP

\theoremstyle{plain}
\newtheorem{theorem}{Theorem}[section]
\newtheorem{proposition}[theorem]{Proposition}
\newtheorem{corollary}[theorem]{Corollary}
\newtheorem{lemma}[theorem]{Lemma}

\theoremstyle{definition}

\DeclarePairedDelimiter\ceil{\lceil}{\rceil}
\DeclarePairedDelimiter\floor{\lfloor}{\rfloor}

\newcommand{\cP}{\mathcal{P}}

\newcommand{\X}{\mathsf{X}}

\newcommand{\sN}{\mathsf{N}}

\newcommand{\sM}{\mathsf{M}}
\newcommand{\sn}{\mathsf{n}}

\newcommand{\sZ}{\mathsf{Z}}
\newcommand{\Pb}{\mathbb{P}}

\newcommand{\sS}{\mathsf{S}}

\newcommand{\sB}{\mathsf{B}}

\newcommand{\sX}{\mathsf{X}}
\newcommand{\cT}{\mathcal{T}}

\newcommand{\PoisN}{\tilde{\mathsf{N}}}
\newcommand{\PoisS}{\tilde{\mathsf{S}}}
\newcommand{\PoisTau}{\tilde\tau}
\newcommand{\PoisCN}{\widetilde{\mathcal{N}}}
\newcommand{\law}{\mathsf{law}}

\newcommand{\pref}[1]{\hyperref[#1]{[p. \pageref{#1}]}}

\newcommand{\old}[1]{}

\newcommand{\N}{\mathds{N}}

\newcommand{\R}{\mathds{R}}

\newcommand{\T}{\mathds{T}}

\newcommand{\U}{\mathcal{U}_{\infty}}

\newcommand{\Prob}{\mathds{P}}
\DeclareMathOperator{\Var}{\mathrm{Var}}
\DeclareMathOperator{\Cov}{\mathrm{Cov}}

\newcommand{\E}{\mathds{E}}

\renewcommand{\L}{\mathcal{L}}
\newcommand{\F}{\mathcal{F}}

\newcommand{\cN}{\mathcal N}

\newcommand{\type}{\mathrm t}

\DeclarePairedDelimiterX{\inp}[2]{\langle}{\rangle}{#1, #2}

\newcommand{\eqdist}{%
	\mathrel{\vbox{\offinterlineskip\ialign{%
				\hfil##\hfil\cr
				$\scriptscriptstyle\mathrm{law}$\cr
				\noalign{\kern.2ex}
				$=$\cr
}}}}

\newcommand{\defeq}{\vcentcolon=}

\newcommand{\distto}{%
	\mathrel{\vbox{\offinterlineskip\ialign{%
				\hfil##\hfil\cr
				$\scriptscriptstyle\mathrm{d}$\cr
				\noalign{\kern-.05ex}
				$\to$\cr
}}}}
\newcommand{\findimto}{%
	\mathrel{\vbox{\offinterlineskip\ialign{%
				\hfil##\hfil\cr
				$\scriptscriptstyle\mathrm{f.d.}$\cr
				\noalign{\kern-.05ex}
				$\to$\cr
}}}}

\newcommand{\Probto}{%
	\mathrel{\vbox{\offinterlineskip\ialign{%
				\hfil##\hfil\cr
				$\scriptscriptstyle\Prob$\cr
				\noalign{\kern-.05ex}
				$\to$\cr
}}}}
\newcommand{\TVto}{%
	\mathrel{\vbox{\offinterlineskip\ialign{%
				\hfil##\hfil\cr
				$\scriptscriptstyle\mathrm{TV}$\cr
				\noalign{\kern-.05ex}
				$\to$\cr
}}}}

\newcommand{\ind}[1]{\mathds{1}_{\{#1\}}}

\title{An epidemic model in inhomogeneous environment}
\author{Daniela Bertacchi,  J\"urgen Kampf,  Ecaterina Sava-Huss,
 and Fabio Zucca}

\date{\today}
\begin{document}
\maketitle						

%\date{\today}
 \begin{abstract}
The current work deals with an epidemic model on the complete graph $\mathbb{K}_n$ on $n$ vertices in a non-homogeneous setting, where the vertices may have distinct types. Different types differ in the probability of getting infected, 
and/or in the 
%finite 
capacity of infecting other vertices. This generalizes the model in \cite{comets14}. We prove in Theorem \ref{t:LLN} and Theorem \ref{thm:clt} laws of large numbers and 
central limit theorems for the  the total duration of the process and for the number of infected vertices, respectively, when $n\to\infty$. %\textcolor{blue}{The referee suggested to mention this: 
By coupling the epidemic model with a Poisson process, we also obtain continuous-time counterparts of the above-mentioned limit results.
%}
Moreover, we also prove that when all individuals have the same spread capacity, 
% (that is, the same law), 
then a population with inhomogeneous susceptibility is less affected by the epidemics than a homogeneous population.
 \end{abstract}

\noindent {\bf Keywords}: multitype Galton-Watson process, coupon collector, branching process, limit theorems, Markov chain, Poisson process, random trees.

\noindent {\bf AMS subject classification}: 60J10, 60J80, 60F17. 
%\begin{small}\tableofcontents\end{small}

%\begin{small}
%\tableofcontents
%\end{small}

\section{Introduction}\label{sec:intro}

We consider an epidemic model in an inhomogeneous environment, where a virus arrives from outside the system and infects one individual of a population.
% -- the epidemic is caused by a virus, say.
 Then, the infected individual attempts to infect some other individuals of the population, and these individuals, in turn, try to infect other individuals, and so on. We start with a population of $n$ individuals, that will be encoded in the vertices of the complete graph $\mathbb{K}_n$ on $n$ vertices. The virus can circulate along the edges of $\mathbb{K}_n$. When we refer to a vertex of a graph, we think of it as an individual of the population.

Once a target vertex $u$ is hit by an infection attempt, two things may happen.  Either $u$ is visited for the first time, then it becomes infected and makes itself $\L^{\type(u)}(u)$ attempts to infect other vertices in the graph. The number $\L^{\type(u)}(u)$ is called the spread capacity of $u$ and will be introduced in full detail later. Or the vertex $u$ has been infected before, then it has antibodies, and it has already revealed its spread capacity, so nothing happens, but just one piece of the spread capacity of the source vertex is lost. The process stops when either all vertices have been infected, or there are non-infected vertices, but the total spread capacity is exhausted, which happens when too much capacity was wasted on vertices with antibodies.

A mathematically equivalent problem is that of information transmission in a network, as used in  \cite{comets14,CGPV16}: the vertices of the graph are computers and when one of these computers receives a message, it spreads the message to a certain number of further computers, but it does so only if it has received the message for the first time. 

During this work we shall refer to our model as the \textit{virus spread model} or \textit{infection model}. 
We recall that processes where individuals after infection become immune are often called SIR models (susceptible-infected-recovered) -- see \cite{BrittonPardoux19} for an introduction to this model and to other epidemic models.
There is a vast literature on SIR models. On infinite graphs the aim is usually to study the shape of the set of the sites reached by
infection (see for instance \cite{CoxDurrett88}, \cite{Zhang93}, \cite{AndjelSaada2005}) or to identify  the critical parameter below which
the epidemics dies out in a finite time (in these cases there is a parameter which tunes the intensity of the epidemics, see \cite{Lalley2014}).
In the present paper the graph is finite, hence the epidemics ends almost surely in finite time.
Natural questions in this context are how long the whole epidemic lasts, 
whether all vertices are infected during the epidemic and, if not, how large the number of infected/non-infected vertices is. 

A slightly different, well-studied epidemic model is the so-called frog model, where infected individuals randomly walk on the graph and infect all the susceptible individuals they meet on their way. 
When the underlying graph is the complete graph, the frog model is a particular case of our model.
In \cite{Kurtz08} this model was considered on complete graphs where all vertices are equally likely to be infected and they all have the same deterministic spread capacity. 
In that paper individuals are assumed to have a deterministic number $L$ of ``lives'' and one life is lost when the infection
attempt is unsuccessful. This translates into our language by assuming that each individual has a spread capacity equal to $L+1$.
The authors derive laws of large numbers and central limit theorems for the number of individuals eventually infected and law of large numbers for the number of infection attempts. 
The frog model on complete graphs has recently been studied in \cite{Leben19} (with geometric lifespans and simultaneous infection attempts), and \cite{Zhukovskii14} (where simultaneous infection attempts are allowed). 

In \cite{comets14} the model was generalized to random spread capacities and laws of large numbers, central limit theorems and large deviation results are derived. Later on, in \cite{CGPV16} these results have been generalized from complete graphs to Erd\"os-R\'enyi graphs; in particular in \cite{CGPV16} it is shown that the model behaves in the limit roughly like the one on complete graphs from \cite{comets14} by using a coupling argument. 
However, in \cite{comets14} and the other works mentioned so far the environment was homogeneous in the sense that all vertices had the same (distribution of the) spread capacity and the same probability of getting infected. As we all know, such assumptions are not very realistic, e.g.\ the probability to infect a member of the same family is much higher than infecting the member of another family, people with many social contacts have a higher probability of being infected and will infect more people once they are infected. Finally, also factors like age or other diseases can increase the probability of being infected. So, a straightforward generalization of the model introduced in \cite{comets14} is to consider $J$ different types of vertices with distinct spread capacities and infection probabilities, for some fixed number $J\in \N$.  The type of a vertex will be represented by a \emph{type function} $\type:\mathbb{K}_n\to\{1,2\ldots,J\}$ such that, for each vertex $v$, $\type(v)=i\in \{1,2\ldots,J\}$ means that $v$ is of type $i$.

Somehow similarly  to \cite{comets14}, our proof approach will be to consider a multitype Galton-Watson tree and a coupon collector with $J$ different types of coupons on the same probability space as the infection model. In contrast to \cite{comets14}, in our model the coupons will be unequally likely, and thus the success epochs (time needed to collect a new coupon) will not be independent random variables.
Independence was heavily used in \cite{comets14}. To overcome the difficulties arising in the non-independent case, we 
use thinned Poisson processes as time intervals of the infection process.
\vspace{-0.5cm}
\paragraph{Structure of the paper.} In Section \ref{sec:prelim} we rigorously introduce our infection model and multitype Galton-Watson processes. We also show  how to represent such a process as a Markov chain for which we write down the transition probabilities, and how to interpret it as a coupon collector's problem.
Then in Section \ref{sec:pois-pr} we couple the infection model with a Poisson process, and then by using the thinning property of such processes, we get rid of the dependencies produced by the coupons of different types and probabilities of being chosen. Subsequently, in Section \ref{subsec:lim-poi} we prove limit results for the number of infected individuals and for the acquisition times in the continuous-time setting. In Section \ref{sec:inf-mgw}, we look at our model as a random subtree of a multitype Galton-Watson tree with $J$ different types, and we show that as in the homogeneous case, the infection takes place on a macroscopic time level. In Section \ref{sec:lln} we prove law of large numbers for the total duration of the process $\tau_n$ and for the total number $\mathcal{N}_{\tau_n}$ of individuals infected by time $\tau_n$ respectively, while in Section \ref{sec:clt} we prove central limit theorems for the same quantities. 
Finally, in Section \ref{sec:final}, we compare the homogeneous case with the inhomogeneous one. In particular we prove that when the spread capacities of the individuals are identically distributed, the epidemics in the inhomogeneous enviroment (that is, with non-uniform susceptibilities) lasts less and involves a smaller fraction of the population, when compared to the case with uniform susceptibility. On the other hand, if also the spread capacities are allowed to have different laws, then it may happen that a population with non-uniform
susceptibility is less affected by the epidemics (for instance when more susceptible individuals have smaller expected spread capacity) or it may be more affected (for instance when more susceptible individuals have larger expected spread capacity).

\section{Preliminaries}\label{sec:prelim}

Let $(\Omega,\F,\Pb)$ be a generic probability space on which all our random variables and processes are defined.
\vspace{-0.5cm}
\paragraph*{Conventions and notations.}
The number $J\in\N$ will be reserved for the number of types, and we will sometimes write $[J]:=\{1,2,\ldots,J\}$.
We use the same notation $[n]:=\{1,2,\ldots,n\}$ for any natural number $n\ge1$.
All our vectors will be regarded as column vectors, and for a vector $\mathsf{u}\in \R^J$  (or infinite dimensional) the notation $\mathsf{u}=(\mathsf{u}_1,\ldots\mathsf{u}_J)$ is purely for convenience. When we deal with row vectors, it will be either mentioned explicitly, or it will be clear from the context. 
For $i\in [J]$, denote by  $\mathsf{e}_i=(0,\ldots,1,\ldots,0)$ the $i$-th standard basis vector in $\R^J$ that has $1$ in the $i$-th coordinate, and all the other coordinates are $0$. 
\vspace{-0.5cm}
\paragraph*{Stochastic processes.} For any discrete-time processes $(X_t)_{t\in\N}$,  $(X_t)_{t\geq 0}$ stands for\\ $(X_{\floor{t}})_{t\geq 0}$, where $\floor{t}$ stands for the floor function. We may also use the ceiling function $\ceil{\cdot}$. For any interval $I$ in $\R$, define the Skorokhod space $\mathcal D(I,\R)$ as the space of c\'adl\'ag functions (right continuous with left limits) from $I$ to $\R$. As usual, we endow $\mathcal{D}(I,\R)$ with the standard $\mathbf{J}_1$ topology introduced by Skorokhod (see \cite{Bill99} or  \cite{jacod-shiryaev}). We remind that convergence in the  $\mathbf{J}_1$ topology implies uniform convergence on compact sets when the limit function is continuous.

\subsection{Virus spread model}\label{sec:virus-model}

%Let $n\in \N$ and consider 
%the complete graph $\mathbb K_n$ on $n$ vertices. We would like to have different types of vertices, for instance we can think that some vertices belong to the same family, or they have a specific property, and for this reason we introduce, as in the above case of multitype branching processes, a \emph{type function} $\type:\mathbb{K}_n\to\{1,2\ldots,J\}$, which is a deterministic function defined on the vertices of $\mathbb{K}_n$, such that, for each vertex $v$, $\type(v)=i\in [J]$ means that $v$ is of type $i$. While in the case of multitype Galton-Watson trees $(\T,\type)$, the type function $\type$ was random and constructed with the help of the offspring distribution matrix $L$, here the type function is deterministic, and each vertex in $\mathbb{K}_n$ comes at the beginning with its own type.
%We shall use the same notation $\type$, for both those functions, and it will be clear from the context to which function we are referring to.

We take $\mathbb{K}_n$, the complete graph on $n$ vertices, as the base state space where the infection takes place and recall that, for $J\in \N$, the type of a vertex is represented by the \emph{type function} $\type:\mathbb{K}_n\to\{1,2\ldots,J\}$ that is, for each vertex $v$, $\type(v)=i\in \{1,2\ldots,J\}$ means that $v$ is of type $i$. The type function is deterministic and fixed at the beginning.
For $i\in [J]$, denote by $n_i=|\{v\in\mathbb{K}_n:\ \type(v)=i\}|$ the number of vertices in $\mathbb{K}_n$ with type $i$, so that $n=\sum_{i=1}^J n_i$.
We assume that each vertex of type $i\in[J]$ has been assigned the weight $p_i\in(0,1)$ (the probability) so that the total weight is $1=n_1p_1+n_2p_2+\ldots +n_J p_J$. 
The goal of this paper is to extend part of the results in \cite{comets14}, where all vertices have the same type and equal weights $1/n$, to the non-homogeneous setting of vertices with different types and weights. If $p_1=\ldots=p_J=\frac{1}{n}$, then the proportion $n_i/n$ of vertices of type $i$ in the graph represents the weight of this group. 
We do not restrict to the case of equal probabilities, and the weight of the group of vertices of type $i$ will be denoted  by $\alpha_i:=n_i p_i$, for $i\in[J]$.
\vspace{-0.5cm}
\paragraph*{Description of the virus spread model.} We consider the following stochastic process in discrete time. At time $t=0$, a fixed vertex $v_0\in\mathbb{K}_n$ receives a virus from outside, and it spreads the virus among the vertices of $\mathbb{K}_n$ at random, considering that the spread capacity of each vertex is limited and encoded in the following random variables. Let $\L^1,\ldots,\L^J$ be $J$ independent $\N$-valued random variables, where for $i\in[J]$, $\L^i$ represents the number of infection attempts a vertex of type $i$ can make once infected. Furthermore,  for $i,j\in [J]$ denote by $L^{(i,j)}$ the $\mathds{N}$-valued random variable that gives the number of vertices of type $j$ that one infected vertex of type $i$ can in turn infect, so $\L^i=\sum_{j=1}^JL^{(i,j)}$. Below, the random variables $L^{(i,j)}$ will also be used as entries of an offspring distribution matrix $L$ in a multitype Galton-Watson process.

For every $i\in[J]$, let $(\L^i(u))_{\{u\in\mathbb{K}_n:\ \type(u)=i\}}$ be a family (of length $n_i$) of i.i.d.\ copies of $\L^i$.  For a vertex $u$ of type $i$, the random variable $\L^i(u)$ represents \textit{the spread capacity} of $u$, and has the same distribution as $\L^i$. For technical reasons we also need $\L^i(u)$ to be defined if $\type(u)\ne i$, but these random variables are void of meaning. We write 
$\L:=(\L^1,\L^2,\ldots,\L^J)$ for the vector of spread capacities. % so that the sequence of random vectors $\left(\L^1(u_1),\L^2(u_2),\ldots,\L^J(u_J)\right)$ indexed over all $(u_1,u_2,\ldots,u_J)$ with $\type(u_1)=1,\ldots,\type(u_J)=J$ is a sequence of i.i.d.\ copies of $\L$. %We assume $\mathbb{P}[\L^1+\dots +\L^J= 1]<1$. \textcolor{blue}{Should we say the same for the probability of 0 children or is it trivial?}

Notice that the offspring distribution $(L^{(i,j)})_{i,j\in[J]}$ we can realize is far from arbitrary. In fact, given the event $\{\L^i(u)=l\}$ that vertex $u$ makes $l\in\N$ infection attempts, the types of the vertices hit by these $l$ attempts are i.i.d.\ with probabilities $\alpha_1, \dots, \alpha_J$, and thus the numbers $(L^{(i,1)}(u), \dots, L^{(i,J)}(u))$ are multinomial distributed with parameter $l$ and $\alpha_1, \dots, \alpha_J$, i.e.\
\[ \mathbb{P}\big[ L^{(i,1)}(u) =k_1, \dots, L^{(i,J)}(u) = k_J | \L^i(u)=l\big] = { l \choose{k_1, \dots, k_J}} \alpha_1^{k_1} \cdot \dots \cdot \alpha_J^{k_J}. \]
for all $k_1, \dots, k_J \in \N$ with $\sum_{i=1}^Jk_i=l$, where ${0 \choose{0, \dots, 0}}:=1$. Hence
the distribution of the vector $\big( L^{(i,1)}(u) , \dots, L^{(i,J)}(u)  \big)$ is given by 
\begin{equation}\label{eq:off-matrix}
 \mathbb{P}\big[ L^{(i,1)}(u) =k_1, \dots, L^{(i,J)}(u) = k_J \big] =  \mathbb{P}[\L^i(u)=l] \cdot { l \choose{k_1, \dots, k_J}} \alpha_1^{k_1} \cdot \dots \cdot \alpha_J^{k_J},\end{equation}
where $l:= k_1+\dots + k_J$.
With the random variables introduced above, we are now ready to explain how the model works. Suppose that the vertex $v_0\in\mathbb{K}_n$ that receives the virus is of type $i_0\in[J]$, i.e. $\type(v_0)=i_0$. This first vertex reveals its spread capacity $\L^{i_0}(v_0)$, with $\L^{i_0}(v_0)=\sum_{i=1}^JL^{(i_0,i)}(v_0)$ and distributes it randomly among the neighbors, not uniformly, but according to the probabilities $p_i$, $i=1,\ldots,J$. 
One may think of having $\L^{(i_0)}$ independent random walkers at $v_0$ that perform independently one random step according to the transition probability
$p_j$, so that $L^{(i_0,j)}$ walkers reach a vertex of type $i$ 
%\textcolor{blue}
{(type $j$)} for each $j\in [J]$. That is, the spread capacity $L^{(i_0,j)}(v_0)$ is spread only among vertices of type $j$. Recall that $\sum_{u\in\mathbb{K}_n} p_{\type(u)}=\sum_{i=1}^Jn_ip_i=1$. Once another vertex is reached for the first time by such a random walker, this vertex reveals its capacity, and does the same thing as the first vertex $v_0$. If a random walker reaches a vertex infected previously, nothing happens. We continue until there is nothing to be spread around.  
We can describe the above model as a Markov chain in discrete time.
Since the number $n$ of vertices is finite, this process will terminate in finite time, because the spread capacity of each individual (vertex) is assumed to be a.s.\ finite. 

Our main goal is to study the asymptotics, in $n$, for the probability of infecting every site, 
for the proportion of infected vertices of each type before exhaustion and for the total duration of the process.

\subsubsection*{Virus spread model as a Markov chain}

We describe now formally the model introduced heuristically above, as a Markov chain $(\sM_t)_{t\in\N}$, where $\sM_t$ is given by the following 
$(J+1)$-uple in $\mathds{N}^{J+1}$:
$$\sM_t=(\underbrace{\sN^1_t,\sN^2_t,\ldots,\sN^J_t}_{:=\sN_t},\sS_t)=(\sN_t,\sS_t)$$
where $\sN^i_t$, for $i\in [J]$ represents the number of vertices of type $i$ infected by time $t$ and $\sS_t$ represents the total spread capacity (revealed and available for distribution) by time $t$. We suppress the index $n$ in the definition of the Markov chain $\sM_t$ and in the random variables in the 
$(J+1)$-uple,
 but it is clear that all those random variables depend on the total number 
$n$ of available vertices.
Suppose at time $t=0$, the virus reaches a vertex $v_0$ with type $i_0$, that is
\begin{align*}
\sN_0&=\mathsf{e}_{i_0},\quad  \sS_0=\L^{(i_0)}(v_0)\sim \L^{(i_0)}, \quad 
\text{and}\quad \sM_0=(\mathsf{e}_{i_0},
%\textcolor{teal}
{\L^{(i_0)}(v_0)}).
\end{align*}
Given the state $(\sn^1,\ldots,\sn^J,\mathsf{s})$ of the Markov chain $(\sM_t)_{t\in\N}$ at time $t$, the
state at time $t+1$ is given by
\begin{equation}\label{eq:tr-pb-MC}
\begin{cases}
(\sn^1,\ldots,\sn^J,\mathsf{s}-1), & \text{ with probability } \sum_{i=1}^J\sn^i p_i \\
(\sn^1+1,\sn^2,\ldots,\sn^J,\mathsf{s}+\L^{1}_t-1), & \text{ with probability } p_1(n_1-\sn^1)\\
\vdots\\
(\sn^1,\ldots,\sn^i+1,\ldots,\sn^J,\mathsf{s}+\L^{i}_t-1), & \text{ with probability } p_i(n_i-\sn^i)\\
\vdots\\
(\sn^1,\ldots,\ldots,\sn^J+1,\mathsf{s}+\L^{J}_t-1), & \text{ with probability } p_J(n_J-\sn^J)\\
\end{cases}
\end{equation}
as long as $\mathsf{s}\ne 0$. In reality, the process stops if $\mathsf{s}=0$. For technical reasons, we have, however, to assume that it continues, but its values are void of meaning then.  
Remark that $\sN_t=(\sN^1_t,\sN^2_t,\ldots,\sN^J_t)$ is a Markov process itself, the so-called \textit{coupon collector's process}, where the coupons are chosen with unequal probabilities. We denote by $\mathcal N_t=\sum_{i=1}^J\sN_t^i$ the total number of vertices infected -- or coupons collected -- by time $t$.

The sequence of 
the revealed
spread capacities $(\L_t)_{t\in \N}=\left((\L^1_t,\ldots,\L^J_t)\right)_{t\in\N}$ is an i.i.d.\ copy of the random variable $\L=(\L^1,\ldots\L^J)$ as introduced 
at the beginning of this section. 
Since $n$ is finite, and each vertex has a finite spread capacity, this process will terminate in finite time $\tau:=\tau_n$,
\begin{equation}\label{eq:end-time}
\tau_n:=\min\{t\in\N: \sS_t=0\}
\end{equation}
either when all vertices were reached and exhausted their own capacity, or when there is no more capacity to be spread even though some of the vertices were not visited at all. 
We will be interested in the asymptotic behavior of $\tau_n$, and in 
the
 total number of infected vertices at exhaustion $\sum_{i=1}^J\sN^i_{\tau_n}$. It suffices to understand the law of the random vector 
$(\sN^1_{\tau_n},\sN^2_{\tau_n},\ldots,\sN^J_{\tau_n})$ in order to obtain the limit behavior for $\cN_{\tau_n}=\sum_{i=1}^J\sN^i_{\tau_n}$.

\vspace{-0.5cm}
\paragraph*{Acquisition and inter-arrival  times.}
In order to have a better understanding of the joint limit behavior of the process
$\sN_t=(\sN^1_t,\sN^2_t,\ldots,\sN^J_t)$, where  for $i\in[J]$, $\sN^i_t$ represents the number of vertices in $\mathbb K_n$ of type $i$ that have been infected by time $t$, we will also introduce \textit{acquisition times} $T^i_k$ and \textit{inter-arrival times} $\Delta^i_k$: for $i\in [J]$ and $k=1,\ldots,n_i$ denote by  $T^i_k$ the time needed to reach the $k$-th vertex of type $i$, and by $\Delta^i_k$ the time needed, after the $(k-1)$-th vertex of type $i$ has been collected, to collect the $k$-th one, that is
\begin{align*}
T^i_k&=\inf\{t\in \N: \ \sN^i_t=k\}\\
\Delta^i_k&=T^i_k-T^i_{k-1}, \text{
%\textcolor{teal}
{where we put $T^i_0 := 0$,} or equivalently } T^i_k=\sum_{j=1}^k\Delta_j^i
\end{align*}
and $\sN^i_t=\sum_{k=1}^{n_i}\ind{T^i_k\leq t}$.
While in the case when all vertices have the same type and the same probability of being chosen, the inter-arrival times are independent geometrically distributed random variables, when one has different types of vertices with unequal probabilities of reaching them, we do not have independence anymore. %We  overcome this situation by coupling the coupon's collector process with Poisson processes, and using the splitting property of such processes. This method has the advantage of creating independence between random variables in the context of an \textcolor{red}{imaginary}\textcolor{blue}{auxiliary} Poisson process.

\subsection{Multitype Galton-Watson processes}\label{sec:mgw}

It may be useful to view multitype Galton-Watson processes as random subtrees of the Ulam-Harris tree that we define now.
\vspace{-0.5cm}
\paragraph*{Ulam-Harris tree.}
The infinite Ulam-Harris  tree $\U$  is the infinite rooted tree with vertex set $V_{\infty}\defeq \bigcup_{n \in \N_0} \N^n$, the set of all finite strings or words $v_1\cdots v_n$ of positive integers over $n$ letters, including the empty word $\varnothing$ which we take to be the root, and with an edge joining $v_1\cdots v_n$ and $v_1\cdots v_{n+1}$ for any $n\in \N_0$ and any $v_1, \cdots, v_{n+1}\in \N$. Thus every vertex $v=v_1\cdots v_n$ has outdegree $\infty$, and the children of $v$ are the words $v1,v2,\ldots$ and we let them have this order so that $\U$ becomes an infinite ordered rooted tree. We will identify $\U$ with its vertex set $V_{\infty}$, where no confusion arises. For vertices $v=v_1\cdots v_n$ we also write $v=(v_1,\ldots,v_n)$, and if $u=(u_1,\ldots,u_m)$ we write $uv$ for the concatenation of the words $u$ and $v$, that is  $uv=(u_1,\ldots,u_m,v_1,\ldots,v_n)$. The parent of $v_1\cdots v_n$ is $v_1\cdots v_{n-1}$.
Further, if $k \leq m$, we set $u|_k \defeq u_1 \ldots u_k$ for the vertex $u$ truncated at height $k$.
Finally, for $u \in \U$, we use the notation $|u|=n$ for $u \in\N^n$, that is $u$ is a word of length (or height) $n$, that is, at distance $n$ from the root $\varnothing$. The family $\mathcal{T}$ of ordered rooted trees can be identified with the set of all subtrees $\mathds{T}$ of $\U$ that have the property that for all $v\in V(\T$):
$$vi\in V(\mathds{T})\  \Rightarrow \ vj\in V(\mathds{T}), \quad \text{for all }j\leq i,$$
where $V(\T)$ denotes the set of vertices of $\T$. For a tree in $\mathcal{T}$  which is not rooted at $\varnothing$, but at some other vertex $u\in \U$, we write $\T_u$. For two vertices $u,v\in \U$ we denote by $d(u,v)$ their graph distance, that is, the length of the shortest path between $u$ and $v$.
For trees rooted at $\varnothing$, we omit the root and we write only $\T$.

For $\mathcal{T}=\{\text{all ordered rooted subtrees }\T \text{ of }\U\}$ and $J\in \N$, a $J$-type tree is a pair $(\T,\type^{\star})$ where $\T\in \mathcal{T}$ and $\type^{\star}:\T\to\{1,\ldots,J\}$ is a function defined on the vertices of $\T$ which returns for each vertex $v$ its type $\type_{\T}(v)$. We denote by $\mathcal{T}^{[J]}$ the set of all $J$-type trees, and elements of $\mathcal{T}^{[J]}$ will be referred to as  $\T$ without explicitly mentioning the function $\type^{\star}$.

A multitype Galton-Watson processes (MGW), called also multitype branching process, is a natural generalization of a Galton-Watson process, where a finite number of distinguishable types of particles with different
probabilistic behavior are allowed. We use again $J\in \N$ for the number of particle types, and  the particle types will
 be the same as the $J$ different types of vertices in the complete graph $\mathbb{K}_n$ and will be denoted by $\{1,\ldots,J\}$. 
In order to relate the infection model with multitype Galton-Watson processes, we proceed as follows.
We consider the sequence $\left(L^{(i)}\right)_{i\in [J]}$ of $J$ independent random (row) vectors in $\R^J$, with entries  $L^{(i)}=\left(L^{(i,1)},\ldots,L^{(i,J)}\right)$, whose distribution is given by \eqref{eq:off-matrix}, and with these $J$ vectors we build our Galton-Watson tree.
 The vector $L^{(i)}$
represents the offspring distribution vector of a vertex of type $i$, meaning that for $j\in [J]$, the entry $L^{(i,j)}$ represents the number of offspring of type $j$ produced by a vertex of type $i$. Let now $L$ be the $J\times J$ random matrix whose rows are the vectors $L^{(i)}$:
$$L=
%\textcolor{teal}
{\begin{pmatrix} L^{(1)} \\ L^{(2)} \\ \vdots \\ L^{(J)}\end{pmatrix}}
=\begin{pmatrix}
   L^{(1,1)}  & L^{(1,2)} &  \cdots & \cdots  &  L^{(1,J)}\\
   L^{(2,1)}  & L^{(2,2)} & \cdots  & \cdots  & L^{(2,J)} \\
    \vdots    & \vdots    & \vdots &  \vdots & \vdots \\
   L^{(J,1)}   & L^{(J,2)}  & \cdots & \cdots   & L^{(J,J)}
  \end{pmatrix}$$
 
 Starting with the random matrix $L$ as introduced above, we can now define 
multitype Galton-Watson trees  as $\mathcal{T}^{[J]}$-valued random variables, where the type function $\type^{\star}$ is random and defined in terms of the $J\times J$-random matrix $L$. Let $(L(u))_{u\in \U}$ be a family of i.i.d.\ copies of $L$, so $(\L(u))_{u\in\U}$ is a family of i.i.d.\ copies of the random vector $\L=(\L^1,\ldots,\L^J)$ that has independent entries.
For any $i_0\in [J]$, we define the random labeled tree $\T^{\mathsf{MGW}}\in \mathcal{T}^{[J]}$ rooted at $\varnothing$ with the associated type function
$\type^{\star}:\T^{\mathsf{MGW}}\to \{1,\ldots,J\}$ 
recursively as follows:
$$\varnothing\in\T^{\mathsf{MGW}}\quad\text{and}\quad \type^{\star}(\varnothing)=i_0.$$
Now suppose that $u=u_1\ldots u_m\in\T^{\mathsf{MGW}}$ with $\type^{\star}(u)=i$, for some $i\in[J]$. Then 
$$u_1 \ldots u_m k\in \T^{\mathsf{MGW}} \quad\text{iff}\quad k\le L^{(i,1)}(u)+\dots+L^{(i,J)}(u)=\mathcal{L}^i(u)$$ and we assume
\[ |\{k\in\mathds{N} \mid \type^\star(u_1 \dots u_mk)=j\}| =L^{(i,j)}(u) \text{ for all } j\in [J]. \]

\emph{The multitype branching process}  $\sZ_t=(\sZ^1_t, \dots , \sZ^J_t)$ associated with $(\T^{\mathsf{MGW}},\type^{\star})$, and starting from a single particle of type $i_0\in [ J]$ at the root $\varnothing$, that is $\type^{\star}(\varnothing)=i_0$, is defined as: $\sZ_0=\mathrm{e}_{i_0} $ and for $t\geq 1$
\begin{align*}
\sZ^i_t\defeq\#\{u\in\T^{\mathsf{MGW}}:|u|=t\text{ and }\type^{\star}(u)=i\}=
\sum_{v\in\T^{\mathsf{MGW}}:|v|=t-1}L^{(\type^{\star}(v),i)}(v), \quad \text{for } i\in [J],
\end{align*}
that is, $\sZ_t^i$ represents the number of particles of type $i$ in the $t$-th generation, or more precisely the number of vertices $u\in \T^{i_0}$ with $|u|=t$ and $\type^{\star}(u)=i$. 
When referring to multitype Galton-Watson processes we shall always have in mind both $(\sZ_t)_{t\in \N}$ and its genealogical Galton-Watson tree $\T^{\mathsf{MGW}}$, with the corresponding type function $\type^{\star}$, i.e.\ the pair $(\T^{\mathsf{MGW}},\type^{\star})$.

Denote by $m_{ij}=\E\left[L^{(i,j)}\right]$ the expectation of the random variable $L^{(i,j)}$, for all $i,j\in [J]$ and by $M=(m_{ij})_{i,j\in [J]}$ the first moment  or the mean offspring matrix. We have  $\E L=M$. 
Also, for $t\in\mathds{N}$ we denote by $\mathcal{Z}_t$ the total number of offspring in the $t$-th generation, that is
$\mathcal{Z}_t=\sum_{i=1}^J\sZ_t^i$.

%\textcolor{teal}
{We assume, for simplicity, in the following that all entries of $M$ are finite.} If there exists $n\in \N$ such that all entries of $M^n$ are strictly positive, then $(\sZ_t)_{t\in\N}$ is called \textit{positive regular}. If each particle has exactly one child, then $(\sZ_t)_{t\in\N}$ is called singular. 
It is well known that if $(\sZ_t)_{t\in\N}$ is positive regular and %the process is not degenerate in the  sense that $\sum_{l=1}^J L^{(j,l)} = 1$ a.s. for all $j\in[J]$
not singular, then if the spectral radius $\rho(M)\leq 1$, the MGW process $(\sZ_t)_{t\in \N}$ dies out almost surely. If $\rho(M)>1$, then it survives with positive probability $\mathbb{P}\left[\mathsf{Surv^{MGW}}\right]$ and we will denote by $\sigma^{\mathsf{MGW}}=1-\mathbb{P}\left[\mathsf{Surv^{MGW}}\right]$ its complement, where
$$\mathsf{Surv^{MGW}}=\left\{\mathcal{Z}_t=\sum_{i=1}^J\sZ^i_t> 0,\ \forall t\in\mathds{N}\right\},$$
so
$$
\sigma^{\mathsf{MGW}}=1-\mathbb{P}\left[\mathsf{Surv^{MGW}}\right]
\begin{cases}
=1 &,\quad \text{if } \rho(M)\leq 1\\
<1 &,\quad \text{if }  \rho(M)>1.
\end{cases}$$

Note that in the case of multitype Galton-Watson trees $(\T^{\mathsf{MGW}},\type^{\star})$, the type function $\type^{\star}$ is random and constructed with the help of the offspring distribution matrix $L$, while in the infection model the type function $\type$ is deterministic, and each vertex in $\mathbb{K}_n$ comes at the beginning with its own type.
\paragraph{Assumptions.} Our results will be proven under the following assumptions concerning the proportions and the weights of vertices of type $i$, and on the spectral radius $\rho(M)$.

\textsf{Assumption 1}: For every $i\in [J]$, both the proportions $\gamma_i:=\frac{n_i}{n}\in (0,1)$ of vertices of type $i$ and the weights $\alpha_i:=n_ip_i\in (0,1)$  of type $i$ group do not depend on $n$, thus are constant. 
We have $\sum_{i=1}^J\gamma_i=1$.

\textsf{Assumption 2:} 
%\textcolor{teal}
{All entries of $M$ are finite and} $\rho(M)>1$.

\section{Coupling coupon collector's with Poisson processes}\label{sec:pois-pr}

Our proof strategy involves approximating the discrete-time process of collecting coupons (vertices in our case) by a
continuous-time process. Instead of the coupon collector drawing a random coupon  at integer time points, he draws a random coupon  (from the same distribution) at times given by a Poisson process with parameter $1$. The times at which any given vertex is drawn is then a thinned Poisson process, and the Poisson processes associated with different vertices are independent. 
%\textcolor{red}{The times at which
%any particular vertex $v\in \mathbb{K}_n$ is reached will also be a thinned Poisson process. -- The same has already been said in the sentence before.} 
Working in the continuous rather than in the discrete setting simplifies calculations.

Suppose we are collecting coupons at times chosen according to a Poisson process with rate $1$: let $(\mathcal{P}_t)_{t\geq 0}$ be a Poisson process with rate $1=\sum_{i=1}^Jn_ip_i$, so that each Poisson event associated with this process is a sampled coupon (infected individual). Thus $\mathcal{P}_t$ represents the total number of attempts to infect vertices by time $t$ (out of these vertices, some of them may have been infected more than once, and we are only interested in the number of different ones). 

We denote the number of vertices of type $i$ infected by time $t$ by $\PoisN_t^i$, when we refer to the continuous time process, while the notation $\sN_t^i$ is kept for the number of vertices of type $i$ infected up to time $t$ in the discrete time process. The relation between the continuous versus discrete processes can be formalized in the following way:
\begin{align*}
 \PoisN_t^i:=\sN_{\mathcal{P}_t}^i\quad \text{and}\quad \PoisCN_t:=\mathcal{N}_{\mathcal{P}_t}\quad \text{and}\quad  \PoisS_t:=\sS_{\mathcal{P}_t}.
\end{align*}
%\textcolor{blue}{
Moreover,
we denote by 
\[ \tilde\tau_n:= \inf\{t\in\mathbb{R} \mid \mathcal{P}_t \ge \tau_n\} = t_{\tau_n}\]
the total duration of the  Poisson process, and note that $\tau_n=\cP_{\tilde\tau_n}$.
%}
We also use the notation $\PoisN_t=(\PoisN^1_t,\ldots,\PoisN^J_t)$.
At times, it will be useful to enumerate the vertices of  $\mathbb{K}_n$ as $\{v_1,v_2,\ldots,v_n\}$. 
With each vertex $v_k \in\mathbb K_n$, $k=1,\ldots,n$ we associate a Poisson process 
$\mathcal{P}_t(k)$ which represents the number of times $v_k$ has been collected (or 
contacted by an infected individual)
 by time $t$. Then
$(\mathcal{P}_t(k))_{t\geq 0}$ are independent Poisson processes with rates $p_{\type(v_k)}$. More precisely
$\left\{(\mathcal{P}_t(k))_{t\geq 0}: \ v_k\in \mathbb{K}_n, \type(v_k)=i\right\}$
is a sequence of $n_i$ independent Poisson processes with rate $p_i$, for any $i\in [J]$. Therefore, we can write 
$\mathcal{P}_t=\sum_{k=1}^n\mathcal{P}_t(k)$. 
For each $k=1,\ldots,n$ denote by $E_k$ the time of the first event in the Poisson process $(\mathcal{P}_t(k))_{t\geq 0}$ associated with $v_k\in\mathbb K_n$, that is, the first time when vertex $v_k$ was reached and infected. The random variables $(E_k)_{k=1,\ldots,n}$ are independent (because they are associated to independent Poisson processes) and exponentially distributed with rate  $p_{\type(v_k)}$.
We are interested only in the number of different coupons sampled, that is, only in the realization time of the first event in the Poisson processes mentioned above. With this notation and in view of the independence, we have now an alternative way of writing $\PoisN^i_t$ as sum of i.i.d.\ random variables: for $i\in [J]$ we have
\begin{align*}
\PoisN^i_t & =\sum_{v_k\in\mathbb{K}_n; \type(v_k)=i}\ind{E_k\leq t},\quad \text{and  } E_k\sim Exp(p_i)
\end{align*}
and hence
\begin{align*}
\mathbb{E}[\PoisN^i_t]=n_i\mathbb{P}[E_k\leq t]=n_i(1-e^{-p_it}),
\end{align*}
and
$$\Var[\PoisN^i_{t}]=n_i\left[(1-e^{-p_i t})-(1-e^{-p_i t})^2\right]=n_ie^{-p_i t}(1-e^{-p_i t}).$$
With this coupling, we are ready to proceed with limit theorems for the processes\\ $\left((\PoisN^1_{n_1s},\ldots,\PoisN^J_{n_Js})\right)_{s\geq 0}$ and the corresponding acquisition times $\left(\tilde T^1_{\floor{n_1 q}},\ldots,\tilde T^J_{\floor{n_Jq}}\right)_{q\in [0,1)}$
as $n\to\infty$. In both random vectors, each random variable depends on $n$, but we will not write the dependence explicitly in order to keep the notation simpler.

\subsection{Limit theorems for infected individuals and acquisition times}\label{subsec:lim-poi}

For simplicity, we introduce some additional notation.
\begin{itemize}\setlength\itemsep{0em}
\vspace{-0.25cm}
\item We will write $\mathsf{p}=(p_1,\ldots,p_J)$ and $\mathsf{\gamma}:=(\gamma_1,\ldots,\gamma_J)$.
\item  We denote  by $\beta_i=\frac{\alpha_i}{\gamma_i}=np_i$.
\item We write $w_i(s):=\gamma_i(1-e^{-\beta_i s})$, and $\mathsf{w}(s):=(w_1(s),\ldots,w_J(s))$, for $s\in \R$, so we have
\begin{equation}\label{eq:weights}
\inp{\mathsf{w}(s)}{\mathsf{1}}=\sum_{i=1}^J \gamma_i(1-e^{-\beta_i s})=1-\sum_{i=1}^J \gamma_ie^{-\beta_i s},
\end{equation}
where $\mathbf{1}$ is the all ones vector.
\end{itemize}

\begin{proposition}\label{prop:clt-n_i} Suppose that \textsf{Assumption 1} holds.
For every  $i\in [J]$, for the continuous-time coupon collector's process $(\PoisN_t^i)_{t\in \N}$  we have the following weak convergence in the Skorokhod space $\mathcal{D}([0,\infty),\mathbb{R})$ endowed with the standard  $\mathbf{J}_1$ topology:
\begin{align*}
\left\{\frac{1}{\sqrt{n_i}}\left(\PoisN^i_{ns}-n_i\left(1-e^{-\beta_i s}\right)\right);\ s\geq 0\right\}& \stackrel{\law}{\longrightarrow}
\left\{\mathsf{X}^{(i)}_s;\ s\geq 0\right\},\quad \text{ as }n\to\infty.
\end{align*}
where, for every $i\in [J]$,  $(\mathsf{X}^{(i)}_s)_{s\geq 0}$ is a centered Gaussian process with 
$$\Cov\left[\mathsf{X}^{(i)}_s,\mathsf{X}^{(i)}_t\right]=e^{-\beta_i t}\left(1-e^{-\beta_i s}\right)>0,$$
for all $s,t\in [0,\infty)$ with $s\le t$.
\end{proposition}
\begin{proof}
The claim follows from \cite[Theorem 14.3]{Bill99}, since $\PoisN^i_t = \sum_{k=1}^{n_i} \mathbf{1}_{\{E_k\le t\}}$ for i.i.d.\ random variables $E_k\sim Exp(p_i)$, $k=1, \dots, n_i$.
\end{proof}

Since with the Poissonization technique, the processes $(\PoisN^i_t)_{t\in \N}$ are independent (because they are measurable functions of independent Poisson processes) we also have their joint convergence as noticed in the following proposition. 

\begin{proposition}\label{p:CLT_Nt-joint}
Suppose that \textsf{Assumption 1} holds.
For the 
%(continuous-time) 
coupon collector's process $(\PoisN_t)_{t\in \N}$ with $J$ different types of coupons and $\PoisN_t=(\PoisN^1_t,\PoisN^2_t,\ldots,\PoisN^J_t)$,  we have the following (joint) weak convergence in the Skorokhod space $\mathcal{D}([0,\infty),\mathbb{R}^J)$ endowed with the standard  $\mathbf{J}_1$ topology:
\begin{align*}
\left\{\frac{1}{\sqrt{n}}\left(\vec{\sN}_{ns}-n\mathsf{E}_s\right);\ s\geq 0\right\}& \stackrel{\law}{\longrightarrow}
\left\{\mathsf{X}_s;\ s\geq 0\right\}
\end{align*}
where the $J$-dimensional processes above are defined as  
\begin{align*}
\vec{\sN}_{ns}:=\left(\frac{\PoisN^1_{ns}}{\sqrt{\gamma_1}},\ldots,\frac{\PoisN^J_{ns}}{\sqrt{\gamma_J}}\right).
\end{align*}
The expectation vector $\mathsf{E}_s$ is given by 
$\mathsf{E}_s:=\left(\sqrt{\gamma_1}(1- e^{-\beta_1 s}),\ldots,\sqrt{\gamma_J}(1- e^{-\beta_J s})\right)$,
and the $J$-dimensional stochastic process $(\X_s)_{s\geq 0}$ with $\X_s=(\X^{(1)}_s,\ldots, \X^{(J)}_s)$ has independent components $\X^{(1)}_s,\ldots,\X^{(J)}_s$ which are centered Gaussian processes with covariances
$$\Cov\left[\mathsf{X}^{(i)}_s,\mathsf{X}^{(i)}_t\right]=e^{-\beta_i t}\left(1-e^{-\beta_i s}\right)>0, \quad \text{for all } s,t\in [0,\infty) \text{ with } s\le t.$$
\end{proposition}
Since $\PoisCN_t=\sum_{i=1}^J\PoisN^i_t$, an immediate corollary of Proposition \ref{p:CLT_Nt-joint} is the following. 
\begin{corollary}\label{c:FCLT_N}
We have the following weak convergence in the  Skorokhod space $\mathcal{D}([0,\infty),\mathbb{R})$ endowed with the standard  $\mathbf{J}_1$ topology:
$$
 \left\{\dfrac{1}{\sqrt{n}}\left(\PoisCN_{ns}-n \inp{\mathsf{w}(s)}{\mathsf{1}} \right); s\geq 0\right\} 
 \stackrel{\law}{\longrightarrow}
\left\{\mathsf{X}^{\cN}_s;\ s\geq 0\right\}
$$
where $\mathsf{X}^{\cN}_s$ is a centered Gaussian processes with 
$$\mathsf{Cov}[\mathsf{X}^{\cN}_s, \mathsf{X}^{\cN}_t] = \sum_{i=1}^J\gamma_i \cdot (1- e^{-\beta_i s}) \cdot e^{-\beta_i  t},\quad \text{for all } s,t\in [0,\infty) \text{ with } s\le t.$$
\end{corollary}
For the vector of acquisition times processes $\left(\tilde T^1_{\floor{n_1 q}},\ldots,\tilde T^J_{\floor{n_Jq}}\right)_{q\in [0,1)}$ of the continuous time process we obtain a similar behavior. Recall that by $ \tilde T^i_{n_iq}  :=\tilde T^i_{\floor{n_iq}}$ for $i\in [{J}]$, we denote the acquisition time of the $\floor{n_i q}$-th coupon of type $i$.
\begin{proposition}\label{prop:clt-T_i} 
Suppose that \textsf{Assumption 1} holds.
For every  $i\in [J]$,  for the  acquisition times $\tilde T^i_k$ associated with the process $\PoisN^i_t$ of collecting vertices of type $i$, we have the following weak convergence in the Skorokhod space $\mathcal{D}([0,1),\mathbb{R})$endowed with the standard  $\mathbf{J}_1$ topology:
\begin{align*}
 &\left\{ \sqrt{n_i} \left( \frac{1}{n} \tilde T^i_{n_iq} - \frac{1}{\beta_i}\log\frac{1}{1-q}\right);  q\in [0,1)\right\} \stackrel{\law}{\longrightarrow} \left\{ \mathsf{B}^i\big(\sigma^i(q) \big) ;\  q\in[0,1)\right\}
\end{align*}
where, for $i\in [J]$,   $\left(\mathsf{B}^i(\sigma^i(q)\right)_{q\in [0,1)}$ is a standard Brownian motion with  
$\sigma^i(q)=\frac{q}{\beta_i^2(1-q)}$ 
%\textcolor{blue}
{(meaning that the limiting process at time $q$ coincides with the Brownian at time $\sigma^i(q)$).}
\end{proposition}
\begin{proof}
We have the following reciprocal equations that relate the distribution of the number of collected coupons $\PoisN^i_t$ to the distribution of the arrival times $\tilde T^i_s$ of the respective coupons: for every $i=1,2,\ldots,J$
$$\{\PoisN^i_t< s \} = \{\tilde T^i_s>t \},$$
that is $\tilde T^i_s$ is a generalized inverse of $\PoisN^i_t$.
For $i\in [J]$, we write
\[ f_i(q):= \frac{1}{\beta_i} \cdot\log\frac{1}{1-q}, \quad q\in[0,1), \quad \text{and}\quad 
f^{-1}_i(s):=1-e^{-\beta_i s}, \quad s\in[0,\infty),\]
so that  $f^{-1}_i(s)$ is the inverse of $f_i(s)$. We apply a Taylor series expansion to $f_i$ in $\frac{\PoisN^i_{ns}}{n_i}$.
Indeed, there is $q_s\in\left[ \min\{ \PoisN^i_{ns}/n_i,f_i^{-1}(s)\}, \max\{ \PoisN^i_{ns}/n_i,f_i^{-1}(s)\} \right]$ such that
\begin{equation}\label{eq:taylor-exp}
f_i\left( \frac{\PoisN^i_{ns}}{n_i}\right) -s = \underbrace{f_i'(f_i^{-1}(s))}_{=\frac{-1}{\beta_ie^{-\beta_i s}}}\cdot \left(\frac{\PoisN^i_{ns}}{n_i} -f_i^{-1}(s) \right) + \frac{1}{2} f_i''(q_s) \cdot \left(\frac{\PoisN^i_{ns}}{n_i} -f_i^{-1}(s) \right)^2 .
\end{equation}
If we show that for every fixed $i\in [J]$
\begin{equation} \label{eq:comp_conv}
\sup_{s\in C} f_i''(q_s) \cdot \left(\frac{\PoisN^i_{ns}}{n_i} -f_i^{-1}(s) \right)^2\longrightarrow 0, \quad \text{in probability, as }n\to\infty 
 \end{equation}
for every compact set $C\subseteq [0,\infty)$, then according to \cite[Proposition VI.1.17 b)]{jacod-shiryaev} 
\[ f_i''(q_s) \cdot \left(\frac{\PoisN^i_{ns}}{n_i} -f_i^{-1}(s) \right)^2\stackrel{\law}{\longrightarrow} 0 \]
in the Skorokhod topology as $n\to\infty$. So let $C\subseteq [0,\infty)$ be compact, and denote by $s_0$ the supremum of $C$.  Choose $q_1\in(f_i^{-1}(s_0),1)$. 
From Proposition \ref{prop:clt-n_i}    
we also have that 
$\frac{\PoisN^i_{ns}}{n_i} \to f_i^{-1}(s)$ in probability as $n\to\infty$, uniformly on compact sets of $\R_+$, so
$$\lim_{n\to\infty}\mathbb{P}\left[\frac{\PoisN^i_{ns_0}}{n_i}>q_1\right]=0.$$
 Put $M:=f''_i(q_{s_0})=\sup_{s\in C}f_i''(q_s)$. Observe that $f_i''(q)=\frac{1}{\beta_i(1-q)^2}>0$ is monotonically increasing on $[0,1)$.  
 Then 
\[\mathbb{P}\left[M>f_i''(q_1)\right] \le \mathbb{P}\left[\frac{\PoisN^i_{ns_0}}{n_i}>q_1\right]\longrightarrow 0, \]
i.e.\ $M$ is bounded in probability.  
Since the operator which maps a function $x:[0,\infty)\to\mathbb{R}$ to $\sup_{t\in C} x(t)$ is continuous at continuous functions $x:[0,\infty)\to\mathbb{R}$, the continuous mapping theorem in its sharp version (see e.g.\ \cite[Theorem 2.7]{Bill99}) implies  \eqref{eq:comp_conv}.
From  \eqref{eq:taylor-exp} it follows that $\left(f_i\left( \frac{\PoisN^i_{ns}}{n_i}\right) -s\right)$ and  $f_i'(f_i^{-1}(s))\cdot \left(\frac{\PoisN^i_{ns}}{n_i} -f_i^{-1}(s) \right) $ have the same limit law, since the remainder in the expansion converges to 0 in probability, and so Proposition \ref{prop:clt-n_i} yields 
\begin{align*}
\sqrt{n_i}\left(f_i\left( \frac{\PoisN^i_{ns}}{n_i}\right) -s\right)\stackrel{\law}{\sim}f_i'\left(f_i^{-1}(s)\right)\underbrace{\frac{1}{\sqrt{n_i}}\left(\PoisN^i_{ns}- n_i f_i^{-1}(s)\right)}_{ \stackrel{\law}{\longrightarrow}\X^{(i)}_s} \stackrel{\law}{\longrightarrow}f_i'\left(f_i^{-1}(s)\right)\X^{(i)}_s.
\end{align*}
By \cite[Theorem 1]{verwaat72} this is equivalent to 
\[ \left\{ \sqrt{n_i} \Big( \frac{1}{n} \tilde T^i_{n_i f_i^{-1}(t)} - t \Big);  t\geq 0\right\} \stackrel{\law}{\longrightarrow} -\left\{f_i'\left(f_i^{-1}(t)\right) \cdot \X^{(i)}_t; t \geq 0 \right\}, \]
which, in turn is equivalent to 
\[ \left\{ \sqrt{n_i} \Big( \frac{1}{n} \tilde T^i_{n_i q} - f_i(q) \Big);  q\in[0,1)\right\} \stackrel{\law}{\longrightarrow} -\left\{ f_i'(q)\cdot \X^{(i)}_{f_i(q)};\  q\in[0,1) \right\}.  \]
Observe now that, for $p,q\in [0,1)$ with $p\le q$
\begin{align*}
\Cov\left[-f'_i(p)\cdot \X^{(i)}_{f_i(p)}, -f'_i(q)\cdot \X^{(i)}_{f_i(q)}\right] &= f_i'(p) \cdot f_i'(q) \cdot f_i^{-1}(f_i(p)) \cdot \left[1-f_i^{-1}(f_i(q))\right]\\
& =\frac{p}{\beta_i^2 (1-p)}
\end{align*}						
and hence 
\begin{align*} 
-\left\{ f_i'(q)\cdot \X^{(i)}_{f_i(q)};\ q\in[0,1) \right\} &=-\left\{\frac{1}{\beta_i(1-q)}\cdot \X^{(i)}_{-\log(1-q)/\beta_i};\ q\in[0,1)\right\}\\
&= \left\{ \mathsf{B}^i\big(\sigma^i(q) \big) ; \  q\in[0,1)\right\},
\end{align*}
in law in the Skorokhod space, with $\sigma^i(q)=\frac{q}{\beta_i^2(1-q)}$.
Therefore
\begin{align*}
 &\left\{ \sqrt{n_i} \left( \frac{1}{n} \tilde T^i_{n_i q} - \frac{1}{\beta_i}\log\frac{1}{1-q}\right);  q\in [0,1)\right\} \stackrel{\law}{\longrightarrow} \left\{ \mathsf{B}^i\big(\sigma^i(q) \big) ;\  q\in[0,1)\right\}
\end{align*}
and this proves the claim.
\end{proof}

Since for each $i\in [J]$, the random variables $(\tilde T_k^i)$ are measurable functions of $\PoisN^i_k$, and $\PoisN^1_t,\ldots,\PoisN^J_t$ are independent, the random variables $\tilde T^1_{n_1},\ldots,\tilde T^J_{n_J}$ are independent as well, so we also have their joint convergence.

\begin{proposition}\label{p:CLT_T-joint}
Suppose that \textsf{Assumption 1} hold.
For the continuous time acquisition process $(\tilde T_k)_{k\in \N}$ with $\tilde T_k=(\tilde T^1_k,\ldots,\tilde T^J_k)$, we have the following (joint) weak convergence in the Skorokhod space 
%\textcolor{teal}
{$\mathcal{D}([0,1),\mathbb{R}^J)$} endowed with the standard  $\mathbf{J}_1$ topology:
\begin{align*}
\left\{\left(\frac{1}{\sqrt{n}}\vec{T}_{nq}-\sqrt{n}\log\frac{1}{1-q}\mathsf{E}[\tilde T_{nq}]\right);\ q\in [0,1)\right\}& \stackrel{\law}{\longrightarrow}
\left\{\mathsf{B}(\sigma(q));\ q\in [0,1)\right\}
\end{align*}
where the $J$-dimensional processes above are defined as following:
\begin{align*}
\vec{T}_{n q}:=\left(\sqrt{\gamma_1}\tilde T^1_{n_i q},\ldots,\sqrt{\gamma_J}\tilde T^J_{n_J q}\right),
\end{align*}
the expectation vector $\mathsf{E}$ is given by 
$$\mathsf{E}[\tilde T_{nq}]:=\left(\frac{\sqrt{\gamma_1}}{\beta_1},\ldots,\frac{\sqrt{\gamma_J}}{\beta_J}\right),$$
and $\mathsf{B}(\sigma(q)):=(\mathsf{B}^1(\sigma^1(q)),\ldots,\mathsf{B}^J(\sigma^J(q)))$ is a $J$-dimensional standard Brownian motion, 
%\textcolor{blue}
{with independent components} and 
 $\sigma^i(q)=\frac{q}{\beta_i^2(1-q)}$.
\end{proposition}

\section{Virus spread model as a random subtree of a MGW}\label{sec:inf-mgw}

We will now show how to obtain a virus spread model by coupling a coupon collector with a multitype Galton-Watson process (MGW) with $J$ types. We will cut from the multitype Galton-Watson tree $\T^{\mathsf{MGW}}$ nodes that 
correspond to vertices in $\mathbb K_n$ that
are reached by more than one random walker. 
We will construct, on a joint probability space $(\Omega,\mathcal{F},\mathbb{P})$
\begin{itemize} \setlength\itemsep{0em}
\item a multitype Galton-Watson process with offspring distribution matrix $L$, where the row sums are $(\L^1,\ldots,\L^J)$, and $J\in \N$ is the number of types,
\item for each integer $n$, a coupon collector's process with $J$ different types of coupons,
\item for each integer $n$, sequences $(\sS_t)_{t\in \N}$ and
 $(\sN_t)_{t\in\N}=((\sN^1_t,\ldots,\sN^J_t))_{t\in \N}$, with $\sN_0=\mathsf{e}_{i_0}$ for a fixed $i_0\in [J]$, $\sS_0\sim \mathcal{L}^{i_0}$, and such that $(\sN_t,\sS_t)_{t\in \N}$ is a Markov chain with transition probabilities as in \eqref{eq:tr-pb-MC}. 
\end{itemize} 

Notice that the situation here is slightly different from the situation in \cite{comets14}. In \cite{comets14} the Galton-Watson process and the coupon collector were assumed to be independent and then the sequence $(\sN_t,\sS_t)_{t\in \N}$ was constructed out of them
(by pruning the Galton-Watson tree). This is not possible anymore here: the types of the Galton-Watson process and the coupon collector have to match and thus we cannot assume that these processes are independent of each other.

In order to construct the processes mentioned above, we start with a multitype Galton-Watson tree $(\T^{\mathsf{MGW}},\type^{\star})$ where $\type^{\star}$ is the type function in the tree rooted at $\varnothing$ with $\type^{\star}(\varnothing)=i_0$, with offspring distribution matrix $L$, which has independent rows $L^{(i)}=(L^{(i,1)},\ldots,L^{(i,J)})$ and distribution given as in \eqref{eq:off-matrix}. Also $\L^i=\sum_{j=1}^J L^{(i,j)}$ represents the total offspring number of a vertex of type $i$, and the random variables $\L^1,\ldots,\L^J$ are independent but not identically distributed. For the following construction it will be crucial that the types of the children $v_1, \dots, v_l$ of each vertex $v$ of $\T^{\mathsf{MGW}}$ have an exchangeable distribution, i.e.\ the distribution of $(\type^{\star}(v_{\pi(1)}), \dots, \type^{\star}(v_{\pi(l)}))$ does not depend on the permutation $\pi:\{1, \dots, l\}\to \{1,\dots, l\}$. 

Recursively, for $t=0,1,\ldots$ we construct $X(t)\in \T^{\mathsf{MGW}}$ representing the $t$-th tentative infection,
and disjoint subtrees $\cT^{\circ}(t),\cT^{\boxplus}(t),\cT^{\dagger}(t)\subset \T^{\mathsf{MGW}}$.
Note that if $X(t)$ has type $i$, then the  corresponding infection attempt in $\mathbb{K}_n$ is at a vertex $u$ with type $i$.
Moreover, if this infection attempt is successful (that is, $u$ was not infected before), then the further infection attempts made by $u$ are represented by the children of $X(t)$. We define
$$\cT^{\circ}(t)=\bigcup_{i=1}^J\underbrace{\{v\in\cT^{\circ}(t):\ \type^{\star}(v)=i\}}_{:=\cT^{\circ}_i(t)}$$
representing the set of vertices (written as a partition in vertices of types $i\in [J]$, contained in the set $\cT^{\circ}_i(t)$) already infected by time $t$. The set $\cT^{\boxplus}(t)$ is the set of tentative infections scheduled but not performed by time $t$, that we write again as a disjoint union of vertices of type $i$ 
$$\cT^{\boxplus}(t)=\bigcup_{i=1}^J\underbrace{\{v\in\cT^{\boxplus}(t):\ \type^{\star}(v)=i\}}_{:=\cT^{\boxplus}_i(t)}.$$
Finally $\cT^{\dagger}(t)$ represents the set of failed spreads (the set of failures in the coupon collecting process);
we do not partition this part into subsets with different types. At time $t=0$, we start with the following data: for some fixed $i_0\in [J]$ and root $\varnothing$ of type $i_0$

\begin{equation*}
\begin{cases}
X(0)  &=\varnothing \text{ (the root of the MGW tree) with } \type^{\star}(\varnothing)=i_0\\
\cT^{\circ}(0)   & =\{\varnothing\} \text{ where } \cT^{\circ}_{i_0}(0)=\{\varnothing\} \text{ and } \cT^{\circ}_i(0)=\emptyset,\ \forall i\ne i_0\\
\cT^{\boxplus}(0)&=\bigcup_{i=1}^J\underbrace{\{v\in\mathbb{T}^{\mathsf{MGW}}:\ \type^{\star}(v)=i; |v|=1\}}_{:=\cT^{\boxplus}_i(0)}\\
\cT^{\dagger}(0)  &=\emptyset
\end{cases}
\end{equation*}
Given the process  $\left(X(t),\cT^{\circ}(t),\cT^{\boxplus}(t),\cT^{\dagger}(t)\right)$ at time $t$, its value at time $t+1$ can be defined as 
follows.
If $\cT^{\boxplus}(t)$ is non-empty, let 
$$v=X(t+1)=\inf\{w\in \cT^{\boxplus}(t)\},$$
where the vertices of $\T^{\mathsf{MGW}}$ are sorted according to their generation, within the generation lexicographically. 
  The type $j_0=\type^{\star}(v)$ determines the type of the vertex $u\in\mathbb{K}_n$ which is hit by the infection attempt at time $t+1$. Each of the $n_{j_0}$ candidates, i.e.\ each vertex $w\in\mathbb{K}_n$ with $\type(w)=j_0$, has the same probability $\frac{1}{n_{j_0}}$ of being hit. In particular, the decision, whether this infection attempt is successful will be made at random independently of $(\T^{\mathsf{MGW}},\type^{\star})$.  
Let $\mathbf{n}^{j_0}=\left|\cT_{j_0}^{\circ}(t)\right|$.
\begin{itemize}
\setlength\itemsep{0em}
\item With probability $\frac{\mathbf{n}^{j_0}}{n_{j_0}}$,
 the infection attempt fails (resulting from sending the virus to a vertex that was previously infected). Then we set
\begin{equation}
 \begin{cases}
\cT^{\circ}(t+1)&=\cT^{\circ}(t), \text{ i.e. } \cT^{\circ}_{i}(t+1)=\cT^{\circ}_i(t) \text{ for all } i\in[J]\\
\cT^{\boxplus}(t+1)&=\cT^{\boxplus}(t)\setminus\{v\}\\
\cT^{\dagger}(t+1)&=\cT^{\dagger}(t)\cup\{v\}
 \end{cases}
\end{equation}
which means that at time $t+1$ the vertex $v$, that was on standby, is now put in the tree $\cT^{\dagger}(t+1)$ with failures, and we have lost one piece of spread capacity without infecting a new vertex. In this case we define $\sN_{t+1}=\sN_t, \sS_{t+1}=\sS_t - 1$. 
\item With probability $\frac{n_{j_0}-\mathbf{n}^{j_0}}{n_{j_0}}$
 the infection is successful. Then
\begin{equation}
\begin{cases}
\cT^{\circ}_{j_0}(t+1)&=\cT^{\circ}_{j_0}(t)\cup\{v\} \text{ and }\cT^{\circ}_j(t+1)=\cT^{\circ}_j(t),\ \text{for all }j\neq j_0\\
\cT^{\boxplus}(t+1)&=\cT^{\boxplus}(t)\setminus\{v\}\cup\{v_1,\ldots,v_{\L^{j_0}(v)}\}\\
\cT^{\dagger}(t+1)&=\cT^{\dagger}(t)
\end{cases}
\end{equation}
where $v_1,\ldots,v_{\L^{j_0}(v)}$ are the children of $v$ in $\T^{\mathsf{MGW}}$. In this case, in the tree $\cT^{\circ}(t+1)$ of infected vertices, the vertex $v$ of type $j_0$ will be added and  it will be deleted from the set $\cT^{\boxplus}$ of stand-by vertices. In addition, to the later tree $\L^{j_0}(v)=L^{(j_0,1)}(v)+\ldots+L^{(j_0,J)}(v)$ offspring of $v$ with the corresponding types  will be added. We put $\sN_{t+1}=\sN_t+\mathsf{e}_{j_0}, \sS_{t+1} = \sS_t + \L^{j_0}-1$. 
\end{itemize}
We proceed with the algorithm as long as $\cT^{\boxplus}(t)$ is non-empty. If $\cT^{\boxplus}(t)$ is empty then we have no vertices on standby, and the process stops; so we set $\tau_n=t$. The set 
$$\cT^{\circ}(t)=\cT^{\circ}(\tau_n)=\bigcup_{i=1}^J\cT^{\circ}_i(\tau_n)$$
is the set of vertices infected during the entire process.  

 The sequence $(\sN_t^1, \dots \sN_t^J, \sS_t)_{t\in\mathds{N}}$ constructed above, 
up to time $\tau_n$, is a Markov chain with transition probabilities as in \eqref{eq:tr-pb-MC}. This follows from the fact that for any $n\in\mathds{N}$, the type of the vertex $X(t+1)$ in the above construction is distributed as $\mathbb{P}[\type^{\star}(X(t+1))=j] = \alpha_j$ for $j\in[J]$. 
%\textcolor{teal}
{Recall that}
$T^i_k:=\inf\{t\in \N: \ \sN^i_t=k\}$.
We also have that $\tau_n$ is a.s.\ finite and bounded by 
$$\tau_n\leq \max\{T^1_{n_1}, \dots, T^J_{n_J}\} \quad \text{and} \quad |\cT^{\circ}(\tau_n)|=\sum_{i=1}^J|\cT^{\circ}_i(\tau_n)|.$$

Again, for technical reasons and in order to obtain a complete coupon collector's model, we extend the sequences $\sN_t$ and $\sS_t$
beyond the end of the epidemic $\tau_n$.  
Hence for 
$t>\tau_n$
we choose a type 
$\Lambda$ according to $\mathbb{P}[\Lambda=j]=\alpha_j$ for $j\in[J]$.
 Then we let the collection of the next coupon be successful and put $\sN_{t+1}=\sN_t + e_\Lambda$ with probability $\frac{n_\Lambda-\mathsf{N}_t^\Lambda}{n_\Lambda}$ (independent of anything else) and we let it fail with probability $\frac{\mathsf{N}_t^\Lambda}{n_\Lambda}$, in which case we put $\sN_{t+1}=\sN_t$.
 
Consider now $J$ i.i.d.\ (infinite) sequences of vectors, also independent of each other:
 $(\bar{\mathcal{L}}^{i}_l)_{l\geq 1}$ with law $\mathcal{L}^{i}$, $i=1,2,\ldots,J$ and define for $l=1,\ldots,n_i$ 
\begin{equation}\label{eq:rv-li}
\mathcal{L}^{i}_l=
\begin{cases}
\mathcal{L}^{i}(X(T^i_l)) & \text{ if } l\leq |\cT^{\circ}_i(\tau_n)|\\
\bar{\mathcal{L}}^{i}_l   & \text{ if } l> |\cT^{\circ}_i(\tau_n)|.
\end{cases}
\end{equation}
This means that $\mathcal{L}^{i}_l$ is the number of children of the vertex which represents the $l$-th infected vertex of type $i$, if in the infective process we reached at least $l$ vertices of type $i$; otherwise $\mathcal{L}^{i}_l$ is the $l-th$ element of the sequence $\bar{\mathcal{L}}^{i}_l$.
%and like usual, we denote by $\L^i_l$ the sum of the entries in the vector $L^{(i)}_l$. 
We also use the following notation: for $i\in[J]$, let $\mathcal{R}^i_t$ be the spread capacity revealed by vertices of type $i$ up to time $t$, that is:
\begin{equation}\label{eq:inf-type1}
\mathcal{R}^i_t=\sum_{l=1}^{\sN^{i}_t}\L^i_l,
\end{equation}
%Also $R_t=(R_t^1,\ldots,R_t^J)$ is a $J\times J$ matrix
and as before, we also write  $\tilde{\mathcal{R}}_t^i=\mathcal{R}_{\mathcal{P}_t}^i$.
The total capacity available at time $t\in\mathds{N}$  is
\begin{equation}\label{eq:s_t}
\sS_t=\underbrace{\sum_{l=1}^{\sN^1_t}\L^1_l}_{=\mathcal{R}^1_t}+\ldots+\underbrace{\sum_{l=1}^{\sN^J_t}\L^J_l}_{=\mathcal{R}^J_t}-t=\underbrace{\sum_{i=1}^J \mathcal{R}^i_t}_{:=\mathcal{R}_t}-t.
\end{equation}
%\end{proof}

Note that, by construction, for $t\le\tau_n$, we have:
$$\cN_t=\sum_{i=1}^J\sN_t^i=|\cT^{\circ}(t)|=\sum_{i=1}^J\underbrace{|\cT^{\circ}_i(t)|}_{=\sN^i_t}\quad \text{and}\quad 
\sS_t=|\cT^{\boxplus}(t)|.$$
Remark that only $\cT^{\circ}(t)$ is a connected random tree, as subtree of $\mathds{T}^{\mathsf{MGW}}$, 
but $\cT^{\circ}_i(t)$, $i\in [J]$ are not connected, they are only subsets of vertices of  $\cT^{\circ}(t)$.

\begin{proposition}\label{prop:var-li}
For each $i\in[J]$ and $l\in[n_i]$ the random variable $\L^i_l$ has law $\mathcal{L}^{i}$. 
The random variables $(\mathcal{L}^i_l)_{1\leq l\leq n_i, 1 \leq i \leq J}$  are independent (of each other) and also independent 
of the acquisition times $(T^i_l)_{1\leq l\leq n_i, 1\leq i \leq J}$. Also
$$\tau_n= \inf\{t\in\mathds{N}: \mathsf{S}_t=\mathbf{0}\}.$$
\end{proposition}
%\begin{remark} It seems rather obvious that if we remove some random variables from an i.i.d.\ collection of random variables, the i.i.d.\ property is preserved, if we decide for each random variable irrevocable whether to keep it or to remove it before we know its realisation. This fact would simplify the following proof considerable. However, we do neither know a proof nor a reference for this fact, so we cannot use it.
%\end{remark} 
\begin{proof} 
The formula for $\tau_n$ follows from \eqref{eq:s_t} and from the fact that $\tau_n$ is finite. 

%The sequence $(\mathcal{L}^i_l)_{1\leq l\leq n_i}$ is iid by the construction in equation \eqref{eq:rv-li}. \textcolor{teal}{No! It is not that simple. We start with i.i.d.\ random variables $\L^i(w), w\in \mathcal{U}_\infty,$ then remove some elements of the sequence (and renumber the remaining ones such that there are no gaps). But an i.i.d.\ sequence may loose the i.i.d.\ property if you remove some elements and renumber. We have to give an argument, why this does not happen here.}\\
%\textcolor{red}{Independence is not straightforward, but it follows from the fact that the "choice" of the vertex which is the $l$-th coupon, is independent of the original sequence of $\mathcal L$. We can try to say something like this, or mimic what was said in the original paper of Comets.}

For the other claim, consider the system $\mathcal{H}$ of functions $h:\{1, \dots, n\} \to \{1, \dots, J\}$ such that
\[ |\{k\in[n] \mid h(k)=j \} | = n_j, \quad j\in [J]. \]
Associate with each $h\in\mathcal{H}$ a further function
\[ \bar h: [J] \times [n] \to [n] \cup \{\infty\},\ (i,k)\mapsto \inf\{l\in[n] \mid |\{ l'\in [l] \mid h(l')=i\}| \ge k\} \]
so that $\bar h(i,k)$ represents the $k$-th time that $h$ takes the value $i$. 
Furthermore put
\[ H_h=\bigcap_{i_1,i_2,l_1,l_2\colon \bar h(i_1,l_1) <\bar h(i_2, l_2)}\left\{ T_{l_1}^{i_1} < T_{l_2}^{i_2}  \right\} \cap \bigcap_{i_1,i_2,l_1,l_2\colon \bar h(i_1,l_1) >\bar h(i_2, l_2)}\left\{ T_{l_1}^{i_1} > T_{l_2}^{i_2}  \right\}, \]
where $ i_1, i_2$ run in $[J]$ and $l_1, l_2$ run in $[n_{i_1}]$ and $[n_{i_2}]$, respectively.
Write 
$$\mathsf{T}^{[J]}_{\{n_1,\ldots,n_J\}}=\left((T_1^1,\ldots T_{n_1}^1),\ldots,(T_1^J,\ldots T_{n_J}^J)\right).$$
Our aim is to show that
\begin{align}
\mathbb{P} \big[& \L^i_l = \lambda^{(i)}_l \text{ for all } i\in[J], l\in [n_i], \mathsf{T}^{[J]}_{\{n_1,\ldots,n_J\}} \in B_1 \mid H_h \big] \notag\\
& = \prod_{i=1}^J \prod_{l=1}^{n_i} \mathbb{P}[\mathcal{L}^{(i)} = \lambda^{(i)}_l] \cdot \mathbb{P}\big[\mathsf{T}^{[J]}_{\{n_1,\ldots,n_J\}} \in B_1 \mid H_h \big]  \label{e:independent_Hh}\end{align}
for every $h\in\mathcal{H}$, $\lambda^{(i)}_l\in \mathds{N}$, $i\in [J], l\in [n_i]$ and $B_1\subseteq \mathds{N}^n$.
In order to see this put, for a fixed $h\in\mathcal{H}$,
$$K_l := \L^{ h(l)}_{\tilde h(l)}\quad \text{where} \quad   \tilde h(l) := | \{l'\in [l] \mid h(l)=h(l')\} | $$
counts how often the type of $h(l)$ has appeared in $h$ up to time $l$.
With this notation we are then going to show
\begin{equation}\label{e:K=L}  
\mathbb{P}\left[K_m=\lambda \mid K_1, \dots, K_{m-1},\mathsf{T}^{[J]}_{\{n_1,\ldots,n_J\}}\right] = \mathbb{P}\left[\mathcal{L}^{h(m)}=\lambda\right].
\end{equation}
For this set $T_l := T^{h(l)}_{\tilde  h(l)}$, and for $0=k_1<k_2<\dots<k_m, m\in\mathds{N}, w\in \mathcal{U}_{\infty}$, put also
\[ A = \{(T_1, \dots, T_m) = (k_1, \dots, k_m) \} \cap \{X(k_{m-1})=w \} \cap \{ T_m \le \tau_n \}. \]
Then, on $A$, $X(T_m)=X(k_m)$ is $\mathcal{F}$-measurable, where $\mathcal{F}:=\sigma(K_1, \dots, K_{m-1})$. In order to see this we notice that an easy induction argument yields that
\[(X(t), \cT_1^\circ(t), \dots, \cT_J^\circ(t), \cT^\boxplus(t), \cT^\dagger(t))\]
is $\mathcal{F}$-measurable for $t=1, \dots, k_{m-1}$ on $A$ in the sense that the event
\begin{align*}
A &\setlength\itemsep{0em}\cap \{ X(t)=w_1\} \cap \bigcap_{i=1}^J \{ w_{\nu_{i-1}+1}, \dots, w_{\nu_i} \in \cT^\circ_i(t) \} \\
&\cap \{ w_{\nu_{J}+1}, \dots w_{\nu_{J+1}} \in \cT^\boxplus(t) \} 
 \cap \{ w_{\nu_{J+1}+1}, \dots, w_{\nu_{J+2}}\in \cT^\dagger(t) \}
\end{align*}
is measurable w.r.t.\ the trace-$\sigma$-algebra of $\mathcal{F}$ on $A$ for any numbers $1=\nu_0\le \nu_1 \le \dots \le \nu_{J+2}$ and $w_1, \dots, w_{\nu_{J+2}}\in \mathcal{U}_{\infty}$. 
In particular,
\[ X(k_m) = \inf\{w\in \cT^\circ(k_m-1)\} = \inf\{w\in \bigcup_{i=1}^J \cT_i^\circ(k_m-1)\} \]
is $\mathcal{F}$-measurable on $A$. Notice that $X(k_m) > X(k_m-1)=w$.
Next we will show that the acquisition times $T^i_l, i\in[J], l\in[n_i]$ are independent of $K_1, \dots, K_n$.
For this we first observe that on $\{t \le \tau_n \}$ the sequence $\left(\type^{\star}(X(1)), \type^{\star}(X(2)), \dots, \type^{\star}(X(t))\right)$ is independent of $(K_1, \dots, K_n)$. Indeed,
\begin{align*}
&\mathbb{P}\big[ \type^{\star}(X(t)) = j \mid  t \le \tau_n, \type^{\star}(X(1)), \dots, \type^{\star}(X(t-1)), K_1, \dots, K_n \big]\\
&= \sum_{\nu=1}^\infty \sum_{\mu=1}^\nu \sum_{t'=0}^t \mathbb{P}\left[ X(t)=X(t')\mu, \L(X(t'))=\nu, \type^{\star}(X(t')\mu)=j \mid E \right]\\
&=  \sum_{\nu=1}^\infty \sum_{\mu=1}^\nu \sum_{t'=0}^{t-1} \begin{aligned}[t] \mathbb{P}\left[ X(t) = X(t')\mu \mid E\right] \cdot & \mathbb{P}[ \mathcal{L}(X(t'))=\nu \mid X(t) = X(t')\mu, E \big]\\ \cdot & \mathbb{P}[ \type^{\star}(X(t')\mu)=j \mid X(t) = X(t')\mu, \mathcal{L}(X(t'))=\nu,  E \big]
\end{aligned}
\end{align*}
where by $E$ we have denoted
$$E:=\left\{t\le \tau_n,  \type^{\star}(X(1)), \dots, \type^{\star}(X(t-1)), K_1, \dots, K_n\right\}$$
and $X(t')\mu$ denotes the $\mu$-th child of $X(t')$. Putting
\[ E':= \left\{t\le \tau_n,  \type^{\star}(X(1)), \dots, \type^{\star}(X(t-\mu)), \type^{\star}\left(X(t')1\right), \dots, \type^{\star}(X(t')(\mu-1)), K_1, \dots, K_n\right\} \]
we get
\begin{align*}
 \mathbb{P}&[ \type^{\star}(X(t')\mu)=j \mid X(t)=X(t')\mu, \mathcal{L}(X(t'))=\nu,  E \big]\\
 &=\mathbb{P}[ \type^{\star}(X(t')\mu)=j \mid X(t)=X(t')\mu, \mathcal{L}(X(t'))=\nu,  E' \big]\\
 &=\mathbb{P}\left[ \type^{\star}(X(t')\mu)=j \mid  \mathcal{L}(X(t'))=\nu,  t \le \tau_n, \type^{\star}(X(t')1), \dots, \type^{\star}(X(t')(\mu-1))\right]= \alpha_j
\end{align*}
due to the properties of the multinomial distribution. So we get
\[ \mathbb{P}\left[ \type^{\star}(X(t)) = j \mid  t \le \tau_n, \type^{\star}(X(1)), \dots, \type^{\star}(X(t-1)), K_1, \dots, K_n \right] = \alpha_j. \]

Furthermore, let $U(t) \sim U(0,1)$, $t\in\mathds{N}$, be independent random variables (independent of $(K_1, \dots, K_n)$) such that the transmission at time $t$ fails iff $U(t) \ge (n_j- \mathsf{n}^j)/ n_j$. Then an easy induction w.r.t.\ $t$ shows that on $\{t\le \tau_n\}$ the event $\{T^i_l=t^i_l \text{ for all } i\in[J], l\in[n_i]\}$ is $\sigma\left(\type^{\star}(X(\tilde t)1), \dots, \type^{\star}(X(\tilde t)\mathcal{L}(X(\tilde t))), U(\tilde t) | \tilde t<t\right)$ measurable as long as $\max\{ t^i_l \mid i\in[J], l\in[n_i] \} \le t$. Since this holds for arbitrary $t\in\mathds{N}$ and since the random variables $\type^{\star}(X(\tilde t)1), \dots, \type^{\star}(X(\tilde t)\mathcal{L}(X(\tilde t))), U(\tilde t), \tilde t\in \mathds{N}$ are independent of $K_1, \dots, K_n$, we get that $T^i_l, i\in[J], l\in[n_i]$ are independent of $K_1, \dots, K_n$. 

For every $v\in \mathcal{U}_{\infty}$ with $v>w$, there is $B_2=B_2(v)\in\sigma(K_1, \dots, K_{m-1})$ with $\{X(k_m)=v\} \cap A = B_2 \cap A$. Notice that $\type^{\star}(X(T_m))=h(m)$. So for every $B_3\in\sigma(K_1, \dots, K_{m-1})$ and $B_4\in\sigma\left(\mathsf{T}^{[J]}_{\{n_1,\ldots,n_J\}}\right)$ with $B_3 \cap B_4\subseteq A$ and for every $\lambda\in\mathds{N}$ we get
\begin{align*}
\mathbb{P}\left[\{K_m=\lambda\} \cap B_3 \cap B_4\right] & = \mathbb{P}\left[ \{ \mathcal{L}(X(k_m))= \lambda \} \cap B_3 \cap B_4 \right] \\
&= \sum_{v\in \mathcal{U}_\infty, v>w} \mathbb{P}[ \{\L(v)=\lambda\} \cap \{\type^{\star}(v)=h(m)\} \cap \{ v=X(k_m) \} \cap B_3 \cap B_4 ] \\
&= \sum_{v\in \mathcal{U}_\infty, v>w} \mathbb{P}[ \L(v)=\lambda, \type^{\star}(v)=h(m)] \cdot \mathbb{P}[ B_2 \cap B_3 ] \cdot \mathbb{P}[ B_4 ] \\
&= \mathbb{P}[\mathcal{L}^{h(m)}=\lambda] \sum_{v\in \mathcal{U}_{\infty}, v>w} \mathbb{P}[B_2 \cap B_3 \cap B_4]\\
&= \mathbb{P}[\mathcal{L}^{h(m)}=\lambda] \sum_{v\in \mathcal{U}_{\infty}, v>w} \mathbb{P}[\{ v=X(k_m)\} \cap B_3 \cap B_4]\\
&= \mathbb{P}[\mathcal{L}^{h(m)}=\lambda] \cdot \mathbb{P}[B_3 \cap B_4].
\end{align*}
Summing up over all possible values $k_1, \dots, k_m$ and $w\in \mathcal{U}_\infty$ we can relax the assumption $B_3 \cap B_4 \subseteq A$ to $B_3 \cap B_4 \subseteq \{ T_m \le \tau_n \}$.
Since the sets $B_3\cap B_4$ with $B_3\in\sigma(K_1, \dots K_{m-1})$ and $B_4\in\sigma(\mathsf{T}^{[J]}_{\{n_1,\ldots,n_J\}})$ form an intersection-stable generating system of \\ $\sigma(K_1, \dots, K_{m-1}, \mathsf{T}^{[J]}_{\{n_1,\ldots,n_J\}})$, we get
\begin{equation} \mathbb{P}[ \{K_m=\lambda \} \cap B_5] = \mathbb{P}[ \L^{h(m)} = \lambda ] \cdot \mathbb{P}[B_5]  \label{e:independence_B5} \end{equation}
for all $B_5\in \sigma(K_1, \dots, K_{m-1}, \mathsf{T}^{[J]}_{\{n_1,\ldots,n_J\}})$ with $B_5\subseteq \{ T_m \le \tau_n\}$. Since \eqref{e:independence_B5} is trivial for $B_5\in \sigma(K_1, \dots, K_{m-1}, \mathsf{T}^{[J]}_{\{n_1,\ldots,n_J\}})$ with $B_5\subseteq \{ T_m > \tau_n\}$, we can drop the assumption $B_5 \subseteq \{T_m \le \tau_n\}$ in order to  arrive at \eqref{e:K=L}.
Applying \eqref{e:K=L} inductively, we get
\[ \mathbb{P}[ K_1=\lambda_1, \dots, K_n=\lambda_n, \mathsf{T}^{[J]}_{\{n_1,\ldots,n_J\}} \in B_1]  =  \prod_{m=1}^n \mathbb{P}[\L^{h(m)}=\lambda_m] \cdot \mathbb{P}[ \mathsf{T}^{[J]}_{\{n_1,\ldots,n_J\}} \in B_1] \] 
and this shows \eqref{e:independent_Hh}. Summing up over all $h\in\mathcal{H}$, we get the claimed independence properties. 
\end{proof}

Finally, we can also state in our case of different types of vertices, a result similar to \cite[Proposition 2.1]{comets14}.

\begin{lemma}
If we denote by $\sZ^{\mathsf{MGW}}_{\mathsf{tot}}=\left| V(\T^{\mathsf{MGW}})\right|$
the total number of individuals in the Galton-Watson process, then the full transmission event is asymptotically included in the survival event in the sense that 
$$\mathsf{Trans}_n\subset\{\sZ^{\mathsf{MGW}}_{\mathsf{tot}}\geq n\},\quad \text{where} \quad \{\sZ^{\mathsf{MGW}}_{\mathsf{tot}}\geq n\}\downarrow \mathsf{Surv}^{\mathsf{MGW}},\quad n\to\infty.$$
Conversely, the event of termination at time $o(n)$ in the epidemic model converges to the event of extinction in the multitype Galton-Watson process.
\end{lemma}
\begin{proof}
The first claim follows directly from the construction. The converse claim is a consequence of Lemma \ref{lem:taun-eps}. % and from Proposition \ref{prop:var-li}.
\end{proof}

In the sequel we assume
\textsf{Assumption 2} holds, in order to ensure that the MGW is supercritical and hence $\Pb[\mathsf{Surv}^{\mathsf{MGW}}]>0$. 
{The assumption that all entries of $M$ are finite is a standard assumption in the literature on multitype Galton-Watson processes that helps to avoid technical problems. The assumption that $\rho(M)>1$} is close to being necessary for $\Pb[\mathsf{Surv}^{\mathsf{MGW}}]>0$ -- we have only excluded the case $\mathcal{L}^1+\dots + \mathcal{L}^J=1$ a.s., in which we would get very pathological problems. 

We want to show next, that also in the case of $J$ different types of coupons, the transmission takes place on a macroscopic time level. The proof is more involved than the one in the homogeneous case and uses a completely different idea.
\begin{lemma}\label{lem:taun-eps}
Suppose that \textsf{Assumption 1} and \textsf{Assumption 2} hold. There exists $\epsilon_0>0$ such that, for all
$\epsilon\in (0,\epsilon_0)$:
$$\lim_{n\to\infty}\Pb\left[\tau_n\geq n\epsilon,\mathsf{Surv}^{\mathsf{MGW}}\right]=\Pb\left[\mathsf{Surv}^{\mathsf{MGW}}\right]=1-\sigma^{\mathsf{MGW}}.$$
\end{lemma}
\begin{proof}
Put 
$$\epsilon_1:=\left(1-\frac{1}{\rho(M)}\right)\quad \text{and}\quad \epsilon_0:=\min\{\gamma_1, \dots, \gamma_J\} \cdot \epsilon_1$$ and let $\epsilon\in(0,\epsilon_0)$. Then there exists an $\epsilon'\in(0,\epsilon_1)$ with $\epsilon= \min\{ \gamma_1, \dots, \gamma_J\} \cdot \epsilon'$. We have
\[ \{\tau_n \geq n\epsilon\} \supseteq \{ | \cT^{\circ}(\tau_n)| \ge n\epsilon \} \]
since $| \cT^{\circ}(t)|\le t$ for all $t\in\mathds{N}$. In order to find a lower bound on the probability $\Pb\left[|\cT^{\circ}(\tau_n)|\geq n\epsilon,\mathsf{Surv}^{\mathsf{MGW}}\right]$ we construct a new process $\left(\mathcal{U}^\circ(r), \mathcal{U}^\boxplus(r), \mathcal{U}^\dagger(r)\right)$, $r\in\mathds{N}$, similar to $(\cT^\circ(t), \cT^\boxplus(t), \cT^\dagger(t))$.
To this aim, given a fixed $k$ we let $g_k$ be the last generation in the Galton-Watson process such that the total number of individuals ever born does not exceed $k$, that is
$$g_k=\max\left\{g\in\mathds{N} \bigg | \sum_{t=1}^g\mathcal{Z}_t\le k\right\}.$$
 Then
 \vspace{-0.25cm}
\begin{itemize}
\setlength\itemsep{0em}
\item at each epoch $r$, a new vertex of the GW tree is added to $\mathcal{U}^\circ(r)$, until the process reaches generation $g_k+1$;
\item from generation $g_k+1$, vertices of the GW-tree are added with probability $1-\epsilon'$;
\item moreover, from generation $g_k+1$,
as long as only a small fraction - smaller than $\epsilon'$ - of the available "coupons" of each type
has been collected, the vertices that are added to $\mathcal{U}^\circ(r)$ are chosen from the ones in
 $\cT^\circ(t_r)$ ($t_r$ is a function of $r$) .
\end{itemize}
 \vspace{-0.25cm}
Now let $k=\lceil (\log n)^2 \rceil$ and let
$$k_1:=\mathcal{Z}_{1} + \dots + \mathcal{Z}_{g_k}$$ denote the total number of individuals up to the $g_k$-th generation.
We set
\[\left(\mathcal{U}^\circ(0), \mathcal{U}^\boxplus(0), \mathcal{U}^\dagger(0)\right):= \left(\mathcal{T}^\circ(0), \mathcal{T}^\boxplus(0), \mathcal{T}^\dagger(0)\right) \]
and assume that $(\mathcal{U}^\circ(r), \mathcal{U}^\boxplus(r), \mathcal{U}^\dagger(r))$ has been defined for some $r$. Conditionally on $\mathcal{U}^\boxplus(r)\ne \emptyset$ we define $(\mathcal{U}^\circ(r+1), \mathcal{U}^\boxplus(r+1), \mathcal{U}^\dagger(r+1))$ by picking the element $v\in \mathcal{U}^\boxplus(r)$ of least generation, among elements of same generation the first one in the lexicographic order, and distinguishing the following cases:
 \vspace{-0.25cm}
\begin{itemize}
\setlength\itemsep{0em}
\item If $|v| \le g_k$, i.e.\ $v$ belongs to the first $g_k$ generations, then assume that the transmission is successful, i.e.
\begin{align*}
\mathcal{U}^\circ(r+1) &:= \mathcal{U}^\circ(r) \cup \{v\}\\
\mathcal{U}^\boxplus(r+1) &:= \mathcal{U}^\boxplus(r) \setminus \{v\} \cup \{v_1, \dots, v_{\mathcal{L}^{i}(v)}\} \\
\mathcal{U}^\dagger(r+1) &:= \mathcal{U}^\dagger(r),
\end{align*}   
where $v_1, \dots, v_{\mathcal{L}^{i}(v)}$ are the children of $v$ in the Galton-Watson tree $\cT$. 
\item If $|v|>g_k$ and $v\in\cT^\dagger(t+1)\setminus\cT^\dagger(t)$ for some $t\in\{1,\dots, \tau_n\}$ such that $\sN_t^j\le\epsilon' n_j$ for all $j\in\{1, \dots, J\}$, i.e.\ the transmission in the old $\mathcal{T}$-model is attempted at some medium time and fails, then the transmission fails also in the new $\mathcal{U}$-model, i.e. 
\begin{align*}
\mathcal{U}^\circ(r+1) &:= \mathcal{U}^\circ(r) \\
\mathcal{U}^\boxplus(r+1) &:= \mathcal{U}^\boxplus(r) \setminus \{v\} \\
\mathcal{U}^\dagger(r+1) &:= \mathcal{U}^\dagger(r)\cup \{v\}.
\end{align*}  
\item If $|v|>g_k$, $v\in\mathcal{T}^\circ(t+1)\setminus \mathcal{T}^\circ(t)$ for some $t\in \{1, \dots, \tau_n\}$ such that $\sN_t^j\le\epsilon' n_j$ for all $j\in [J]$, 
then let the transmission be successful with probability $\frac{1-\epsilon'}{(n_i -\sN_t^i)/n_i}$, where $i=\type(v)$, conditionally on everything defined so far. 
\item If either $v\in (\cT^\circ(t+1) \cup \cT^\dagger(t+1))\setminus (\cT^\circ(t) \cup \cT^\dagger(t))$ for some $t\in \{1, \dots, \tau_n\}$ such that $\sN_t^j>\epsilon' n_j$ for some $j\in\{1, \dots, J\}$ or $v\notin \cT^\circ(\tau_n)\cup \cT^\dagger(\tau_n)$, then the transmission fails with probability
 $\epsilon'$
 conditionally on everything defined so far and is successful otherwise.
\end{itemize}
This construction is summarized in Table \ref{T:U-constr}. If in any of these cases $t\in\{1,\dots, \tau_n\}$ with $v\in(\mathcal{T}^\circ(t+1)\cup\mathcal{T}^\dagger(t+1))\setminus (\mathcal{T}^\circ(t)\cup\mathcal{T}^\dagger(t))$ exists, then denote $t+1$ by $t_{r+1}$. 

\begin{table}%
\begin{tabular}{ll||l|l|}
&& \begin{minipage}{4.5cm} transmission in $\cT$-model successful \end{minipage} & \begin{minipage}{3cm}transmission in $\cT$-model fails \end{minipage} \\\hline\hline
generation $\le g_k$ &&\multicolumn{2}{c|}{transmission successful} \\\hline
generation $>g_k$ & \begin{minipage}{2.5cm} $\sN_t^j\le\epsilon' n_j$ for all $j\in [J]$ \end{minipage}& \begin{minipage}{4.5cm} transmission successful with probability $\frac{1-\epsilon'}{(n_i -\sN_t^i)/n_i}$ \end{minipage} & transmission fails\\\hline
generation $>g_k$ & \begin{minipage}{2.5cm}{$\sN_t^j>\epsilon' n_j$ for some $j\in [J]$ \\ or $t$ not defined}\end{minipage} & \multicolumn{2}{c|}{ transmission successful with probability $1-\epsilon'$} \\\hline
\end{tabular}
\caption{Summary of the construction of the $\mathcal{U}$-model}
\label{T:U-constr}
\end{table}

If $\mathcal{U}^\boxplus(r)=\emptyset$, the process stops and we set $\rho_n=r$. 
Put
$$r_0:= \max\left\{ r'\in\{1,\dots, \rho_n\} \mid \text{$t_r$ is defined for all $r\le r'$}\right\}.$$
Then either $r_0=\rho_n$ or for the vertex $v\in\T^{\mathsf{MGW}}$ chosen in the $r_0+1$-st step there is no $t\in\{1,\dots, \tau_n\}$ with $v\in (\cT^\circ(t+1) \cup \cT^\dagger(t+1)) \setminus (\cT^\circ(t) \cup \cT^\dagger(t))$. The latter however means, that at the time step $r_1\in\{1, \dots, r_0\}$, when $v$ was added to $\mathcal{U}^\boxplus(r)$, it was not added to $\cT^\boxplus(t)$, so the transmission in the $\mathcal{T}$-model at time $t_{r_1}$ failed, while the transmission in the $\mathcal{U}$-model at time $r_1$ was successful. This is possible in two cases: either the vertex $v_1$ treated at time $r_1$ belongs to the first $g_k$ generations or $\sN_{t_{r_1}}^j>\epsilon' n_j$ for some $j\in\{1,\dots, J\}$. Now assume that the event $\mathsf{Succ}_{g_k}$, that in the first $g_k$ generations all transmissions are successful in the $\mathcal{T}$-model, occurs. Then $v_1$ belonging to the first $g_k$ generations cannot explain anymore that the transmission in the $\mathcal{U}$-model at time $r_1$ is successful, while the transmission in the $\mathcal{T}$-model at time $t_{r_1}$ fails, and we are left with the case that $\sN_{t_{r_1}}^j>\epsilon' n_j$ for some $j\in\{1,\dots, J\}$. Since $r_0\ge r_1$ we get $\sN_{t_{r_0}}^j>\epsilon' n_j$ for some $j\in\{1,\dots, J\}$. In conclusion we have
\begin{equation*} \mathsf{Succ}_{g_k} \subseteq \{ \rho_n =r_0\} \cup \{\sN_{t_{r_0}}^j>\epsilon' n_j \text{ for some } j\in\{1,\dots, J\} \}. \label{e:rho0_alternative} \end{equation*}
Since $\tau_n\ge t_{r_0}$ by the definition of $r_0$ and since $\mathsf{N}_{t_{r_0}}^j = |\mathcal{T}_j^\circ(t_{r_0})| \le | \cT^\circ(t_{r_0})|$, we get
\begin{align*}
\Pb\left[\tau_n\geq n\epsilon,\mathsf{Surv}^{\mathsf{MGW}}\right] &\ge \Pb\left[| \cT^\circ(\tau_n)| \ge n\epsilon,\mathsf{Surv}^{\mathsf{MGW}}\right]\\
 &\ge \Pb\left[| \cT^\circ(t_{r_0})| \ge n\epsilon,\mathsf{Surv}^{\mathsf{MGW}}\right]\\
& \ge \Pb \left[\sN_{t_{r_0}}^j>\epsilon' n_j \text{ for some } j\in\{1,\dots, J\}\right] \\
&\ge \Pb\left[ \rho_n=\infty, \mathsf{Succ}_{g_k}\right]
\ge \Pb\left[ \rho_n=\infty\right]+ \Pb\left[\mathsf{Succ}_{g_k}\right] -1.
\end{align*}
In order to estimate $\mathbb{P}[\mathsf{Succ}_{g_k}]$, put $\bar \beta:=\max\{\beta_1, \dots, {\beta_J}\}$. Notice that the probability that in one step a new vertex is infected is bounded from below by $1-\bar\beta\cdot j/n$ given that $j$ vertices have already been infected. Then
\[
1-\mathbb{P}[\mathsf{Succ}_{g_k}] \le 1-\left(1-\bar\beta\frac{1}{n}\right) \cdot \dots \cdot \left(1-\bar\beta\frac{k}{n}\right) 
\le 1-\left(1-\bar\beta \frac{k}{n}\right)^k 
 \sim \bar\beta\frac{k^2}{n} \to 0,
\]
since $k^2/n\to 0$ as $n\to \infty$.

Let us now turn to the probability of the event $\{\rho_n=\infty\}$ that the process in the $\mathcal{U}$-model survives forever. For any vertex $v$ of the $g_k$-th generation, let $\mathsf{Surv(}v\mathsf{)}$ denote the event that $v$ has infinitely many descendants 
in the $\mathcal U$-model.
 Clearly $\{\rho_n=\infty\} = \bigcup_{l=1}^{\mathcal{Z}_{g_k}} \mathsf{Surv(}v_l\mathsf{)}$, where $v_1, \dots, v_{\mathcal{Z}_{g_k}}$ are the members of the $g_k$-th generation. Notice that the systems of descendants of 
these vertices form independent Galton-Watson trees that are independent of the process $\mathcal{U}^\circ$ up to time 
$k_1$. 
This is a consequence of the fact that the remainder of an i.i.d.\ sequence after a stopping time is again an i.i.d.\ sequence. 

Observe that the offspring distribution of these Galton-Watson trees is given by $\bar{L}^{(i,j)} = \sum_{k=1}^{L^{(i,j)}} \delta_k^{(i,j)}$, where $\delta_1^{(i,j)}, \dots \delta_{L^{(i,j)}}^{(i,j)}$ are independent random variables with distribution
$$\mathbb{P}[\delta_k^{(i,j)}=1]=1-\epsilon'\quad \text{and}\quad \mathbb{P}[\delta_k^{(i,j)}=0]=\epsilon' \quad\text{for} \quad i,j=1,\dots, J,\, k=1, \dots, L^{(i,j)}.$$
Hence the mean value matrix is $(1-\epsilon')M$, and in particular it has spectral radius bigger than $1$. So the probability of the event $\mathsf{Surv(}v_l\mathsf{)}$, that this Galton-Watson process survives, is positive and independent of $l$ or $n$.  
Hence we get for any $z\in\mathds{N}$ that
\[ \mathbb{P}[\rho_n<\infty] \le \mathbb{P} \left[ \{\mathcal{Z}_{g_k} \le z\} \cup \left(\mathsf{Surv(} v_1\mathsf{)}^\complement \cap \dots \cap \mathsf{Surv(}v_{z+1} \mathsf{)}^\complement\right) \right] \le \mathbb{P}[\mathcal{Z}_{g_k} \le z] + (1-\delta)^{z+1},	\]
where $\delta:=\Pb[\mathsf{Surv(}v\mathsf{)}]>0$. 
Therefore
\[ \lim_{n\to\infty} \Pb\left[\tau_n\geq n\epsilon,\mathsf{Surv}^{\mathsf{MGW}}\right] \ge \lim_{n\to\infty} \Pb[ \rho_n=\infty] \ge 1 - \lim_{n\to\infty} \mathbb{P}[\mathcal{Z}_{g_k} \le z] - (1-\delta)^{z+1}.\]
Since $g_k \to \infty$ as $n\to\infty$ almost surely and $ \mathcal{Z}_g\to\infty$ as $g\to\infty$ almost surely conditioned on $\mathsf{Surv}^\mathsf{MGW}$, where the latter is a consequence of \cite[Section II, Theorem 6.1]{Har64}, we have that the total number $\mathcal{Z}_{g_k}$ of individuals in the $g_k$-th generation tends to infinity as $n\to\infty$ conditionally on $\mathsf{Surv^{MGW}}$ almost surely. So
\begin{align*}
\lim_{n\to\infty} \Pb[\mathcal{Z}_{g_k} \le z] &= \lim_{n\to\infty} \Pb\left[\mathcal{Z}_{g_k} \le z, \mathsf{Surv}^\mathsf{MGW}\right] + \lim_{n\to\infty} \Pb\left[\mathcal{Z}_{g_k} \le z, (\mathsf{Surv}^\mathsf{MGW})^\complement\right] \\
&= \Pb\left[(\mathsf{Surv}^\mathsf{MGW})^\complement\right]
\end{align*}
and thus
\[ \lim_{n\to\infty} \Pb\left[\tau_n\geq n\epsilon,\mathsf{Surv}^{\mathsf{MGW}}\right] \ge \Pb\left[\mathsf{Surv}^\mathsf{MGW}\right]  - (1-\delta)^{z+1},\]
which shows the assertion, since $z\in\mathds{N}$ was arbitrary.\end{proof}

\section{Law of large numbers}\label{sec:lln}

Let $\theta$ denote the unique solution in $(0,\infty)$ of
\begin{equation}\label{eq:ftheta}
f(\theta): = \sum_{i=1}^J\E[\L^i]\underbrace{\gamma_i(1-e^{-\beta_i \theta})}_{:=w_i(\theta)}-\theta=0,
\end{equation}
or written in another way $\inp{\E[\L]}{\mathsf{w}(\theta)}=\theta$, where $\E[\L]=(\E[\L^1],\ldots,\E[\L^J])$ and 
$$\mathsf{w}(\theta)=(\gamma_1(1-e^{-\beta_1\theta}),\ldots,\gamma_J(1-e^{-\beta_J\theta}))$$ 
and $\inp{\E[\L]}{\mathsf{w}(\theta)}$ is the usual scalar product of the two involved vectors. Observe that there is a solution, since $f(0)=0$, $f'(0)=\sum_{i=1}^J \E[\L^i] \gamma_i\beta_i-1=\rho(M)-1>0$ due to \textsf{Assumption 2}, $\lim_{\theta\to\infty} f(\theta)=-\infty$ and $f$ is continuous. Moreover, the solution is unique, since $f''(\theta)= -\sum_{i=1}^J \E[\L^i] \gamma_i\beta_i^2 e^{-\beta_i \theta} <0$ for $\theta \in [0,\infty)$ and therefore $f(\theta)=0$ can have at most two solutions in $[0,\infty)$ and one of these solutions is given by $\theta=0$.

\begin{theorem}\label{t:LLN}
%\textcolor{teal}
{Assume that \textsf{Assumption 1} and \textsf{Assumption 2} hold.} 
\begin{enumerate}
\item Then
$$\frac{\tau_n}{n} \longrightarrow \theta  \ind{\mathsf{Surv^{MGW}}} \quad \text{and} \quad \frac{\tilde\tau_n}{n}\longrightarrow \theta  \ind{\mathsf{Surv^{MGW}}} $$
in probability as $n\to\infty$. 
\item Also
$$\frac{\cN_{\tau_n}}{n}\longrightarrow\inp{\mathsf{w}(\theta)}{\mathsf{1}} \ind{\mathsf{Surv^{MGW}}} \quad\text{and} \quad \frac{\PoisCN_{\PoisTau_n}}{n} \longrightarrow\inp{\mathsf{w}(\theta)}{\mathsf{1}} \ind{\mathsf{Surv^{MGW}}} $$
in probability as $n\to\infty$, where $\cN_t=\sum_{i=1}^J\sN^i_t$ and $\PoisCN_{t}=\sum_{i=1}^J\PoisN^i_{t}.$
\end{enumerate}
\end{theorem}
\begin{proof}
 \underline{Proof of $\frac{\tau_n}{n}\to\theta\ind{\mathsf{Surv^{MGW}}}$ in probability as $n\to\infty$.}\\
We use a similar approach as in \cite[Theorem 2.2]{comets14}.
We have $\mathcal{R}^i_t=\sum_{l=1}^{\sN_t^i}\L^i_l$, where for each $i=1,\dots,J$, $\L^i_l$ are i.i.d.\ distributed like $\L^i$ with $\E[\L^i]<\infty$.
Also $\PoisN^1_t,\ldots,\PoisN^J_t$ are independent and for 
every $i$, $\frac{\PoisN^i_{ns}}{n}\to \gamma_i(1-e^{-\beta_is})=w_i(s) $ in probability as $n\to\infty$ 
%\textcolor{teal}
{(see Proposition \ref{prop:clt-n_i})}, and 
$$\tilde{\mathsf{S}}_t=\mathsf{S}_{\mathcal{P}_t}=\sum_{l=1}^{\sN^1_{\mathcal{P}_t}}\L^1_l+\ldots+\sum_{l=1}^{\sN^J_{\mathcal{P}_t}}\L^J_l-\mathcal{P}_t.$$
Using that $\PoisN^i_t=\sN^i_{\mathcal{P}_t}$ for all $i\in [J]$, we can in turn write
$$\tilde{\mathsf{S}}_t=\mathsf{S}_{\mathcal{P}_t}=\sum_{l=1}^{\PoisN^1_t}\L^1_l+\ldots+\sum_{l=1}^{\PoisN^J_t}\L^J_l-\mathcal{P}_t.$$
Therefore 
\begin{align*}
\frac{\tilde{\mathsf{S}}_{ns}}{n}&=\frac{\mathsf{S}_{\mathcal{P}_{ns}}}{n}=\frac{1}{n}\left(\sum_{l=1}^{\PoisN^1_{ns}}\L^1_l+\ldots+\sum_{l=1}^{\PoisN^J_{ns}}\L^J_l-\mathcal{P}_{ns}\right)\\
&=\underbrace{\frac{\sum_{l=1}^{\PoisN^1_{ns}}\L^1_l}{\PoisN^1_{ns}}}_{\to \E[\L^1]}\cdot \underbrace{\frac{\PoisN^1_{ns}}{n}}_{\to \gamma_1(1-e^{-\beta_1 s})}+\cdots+\underbrace{\frac{\sum_{l=1}^{\PoisN^J_{ns}}\L^J_l}{\PoisN^J_{ns}}}_{\to\E[\L^J]}\cdot \underbrace{\frac{\PoisN^J_{ns}}{n}}_{\to \gamma_J (1-e^{-\beta_J s})}-\underbrace{\frac{\mathcal{P}_{ns}}{n}}_{\to s}\\
&\to \sum_{i=1}^J\E[\L^i]\gamma_i(1-e^{-\beta_i s})-s=\inp{\E[\L]}{\mathsf{w}(s)}-s
\end{align*}
in probability, uniformly on compacts of $\mathbb{R}_{+}$, where in the second line above we have applied the weak law of large numbers 
for the sequences of involved random variables, and the last line holds in view of the continuous mapping theorem. 
Now, for any $\delta>0$ and $\epsilon>0$, we have
\begin{align*}
\Pb\left[\left|\tau_n-n\theta\ind{\mathsf{Surv^{MGW}}}\right| >n\delta\right]&=\Pb\left[|\tau_n-n\theta|>n\delta,\mathsf{Surv^{MGW}} \right]+\Pb\left[\tau_n>n\delta,\left(\mathsf{Surv^{MGW}}\right)^\complement \right]\\
&\leq \Pb\left[|\tau_n-n\theta|>n\delta,\mathsf{Surv^{MGW}},\tau_n\geq n\epsilon \right]\\
&\quad+ \Pb\left[\tau_n<n\epsilon,\mathsf{Surv^{MGW}} \right]+\Pb\left[\tau_n>n\delta,\left(\mathsf{Surv^{MGW}}\right)^\complement \right].
\end{align*}
On the right hand side of the above inequality, the second term tends to $0$ as $n\to\infty$ in view of Lemma \ref{lem:taun-eps}. The third term vanishes also as $n\to\infty$, since $\tau_n$ is smaller than the extinction time of the multitype Galton-Watson process, which is a.s.\ finite on the extinction event $\left(\mathsf{Surv^{MGW}}\right)^\complement$. Finally, for the first term,
by picking $\epsilon<\theta-\delta$,
we have the following:
\begin{align*}
\limsup_{n\to\infty}&\ \Pb\left[|\tau_n-n\theta|>n\delta,\mathsf{Surv^{MGW}},\tau_n\geq n\epsilon \right]\\
&\leq  \limsup_{n\to\infty}
\Big(
 \Pb\left[n\epsilon\leq \tau_n\leq n\theta-n\delta,\mathsf{Surv^{MGW}}\right]+ \Pb\left[ \tau_n > n\theta + n\delta, \mathsf{Surv^{MGW}} \right]
\Big)\\
&\leq \limsup_{n\to\infty}
\Big(\Pb\left[\min\{\mathsf{S}_{\floor{ns}};\ s\in [\epsilon,\theta-\delta]\}\le0\right] + \Pb\left[\sS_{\floor{n(\theta+\delta)}}> 0\right] \Big) \\
& \leq \limsup_{n\to\infty}
\Big(\Pb\left[\min\{\mathsf{S}_{\mathcal{P}_{ns}};\ s\in [\epsilon/2,\theta-\delta/2]\}\le0\right] +\Pb\left[ \mathcal{P}_{ n \cdot \epsilon/2} > \floor{n\epsilon} \right] \\
& + \Pb\left[ \mathcal{P}_{n \cdot (\theta-\delta/2)} < \lfloor n\cdot (\theta-\delta) \rfloor \right]
 + \Pb\left[\sS_{\mathcal{P}_{n(\theta+\delta/2)}}> 0\right] \\
 &+ \Pb\left[ \mathcal{P}_{n(\theta+\delta/2)} > \floor{ n\cdot(\theta+\delta) } \right] \Big) = 0.
%\inf\{\mathsf{S}_{{ns}};\ s\in [\epsilon,\theta-\delta]\}=0\right] + \Pb\left[\inf\{\mathsf{S}_{{ns}};\ s\in [\theta-\delta,\theta+\delta]\} \ne 0\right] = 0.
\end{align*}
For the second inequality we have used that since $\tau_n= \min\{ t\in\mathds{N} \mid \mathsf{S}_t=0 \}$ we have $\tau_n\le t_0$ if and only if there is $t\le t_0$ with $\mathsf{S}_t=0$. To see the third inequality observe that $\mathcal{P}_{ns},\, s\in[ \epsilon/2, \theta-\delta/2]$ attains with probability one every integer from $\mathcal{P}_{n \epsilon/2}$ through $\mathcal{P}_{n \cdot(\theta-\delta/2)}$. The final limit relation holds, since $\frac{\tilde{\mathsf{S}}_{ns}}{n}=\frac{\mathsf{S}_{\mathcal{P}_{ns}}}{n}$ converges in probability to $\inp{\E[\L]}{\mathsf{w}(s)}-s$ uniformly on compacts of $\mathbb{R}_{+},$ 
and, moreover, $\inp{\E[\L]}{\mathsf{w}(s)}-s>0$ for $s\in (0, \theta)$ and $\inp{\E[\L]}{\mathsf{w}(s)}-s<0$ for any $s\in(\theta,\infty)$. This completes the proof of the first part.
 
\underline{Proof of $\frac{\tilde\tau_n}{n}\to\theta\ind{\mathsf{Surv^{MGW}}}$ in probability as $n\to\infty$.}\\
We write  
$\tilde\tau_n = \sum_{k=1}^{\tau_n} E_k$ for independent, exponentially distributed random variables $E_k\sim \mathsf{Exp}(1)$ with rate $1$ and set $\theta':=\theta  \ind{\mathsf{Surv^{MGW}}}$. 
Pick any $\epsilon>0$. We first split the event $\left[\left|\frac{\PoisTau_n}{n} -\theta'\right|>\epsilon\right]$ intersecting it with 
$A_n:=\left[\left|\frac{\tau_n}{n}-\theta'\right|\le\frac{\epsilon}{9}\right ]$ and its complement.
Note that on $A_n$, $\tilde\tau_n$ differs from $\sum_{k=1}^{\floor{n\theta'}}E_k$ by a random variable which is the sum of at most $\ceil{n\epsilon/9}$ exponentials of rate 1. Let $W_n$ be this variable (with its sign) so that 
 $\tilde\tau_n=\sum_{k=1}^{\floor{n\theta'}}E_k+W_n$.
\[
\mathbb{P}\left[\left|\frac{\PoisTau_n}{n} -\theta'\right|>\epsilon\right] \le \mathbb{P}\left[A_n^\complement\right ] + \mathbb{P}\left[ \left|\frac{1}{n} \sum_{k=1}^{\lfloor n\theta' \rfloor} \big(E_k-1 \big)+\frac{\floor{n\theta'}}n-\theta'+\frac{W_n}n\right|>\epsilon \right].
\]
We already know that the first term vanishes, so we concentrate on the second one, which we call $\mathbb P(B_n)$.
We split the event by intersecting with $C_n:=\left[\left|W_n/n\right|\le\frac{\epsilon}{3}\right ]$ and its complement. By the weak law of large numbers, $\mathbb P[C_n^\complement]\to0$,
since $|W_n|$ is a.s.\ positive and bounded by a sum of $n\epsilon/9$ terms with mean 1.
Now, use $\theta'-\floor{n\theta'}/n\in[0,1/n]$ to get
\[\begin{split}
\mathbb P[B_n\cap C_n]&=\mathbb P\left[\frac{1}{n} \sum_{k=1}^{\lfloor n\theta' \rfloor} \big(E_k-1 \big)>\epsilon -\frac{W_n}n,C_n\right]\\
&+\mathbb P\left[\frac{1}{n} \sum_{k=1}^{\lfloor n\theta' \rfloor} \big(E_k-1 \big)<-\epsilon+\frac{1}n -\frac{W_n}n,C_n\right].
\end{split}
\] 
For the first term,
\[
\mathbb P\left[\frac{1}{n} \sum_{k=1}^{\lfloor n\theta' \rfloor} \big(E_k-1 \big)>\epsilon -\frac{W_n}n,C_n\right]
\le \mathbb P\left[\frac{1}{n} \sum_{k=1}^{\lfloor n\theta' \rfloor} \big(E_k-1 \big)>\frac23\epsilon\right]\to 0,\quad \text{ as } n\to\infty
\] 
 since by the weak law of large numbers 
$ \frac{1}{n} \sum_{k=1}^{\lfloor n\theta'\rfloor} \big( E_k -1 \big) \to 0$.
For the second term
\[
\mathbb P\left[\frac{1}{n} \sum_{k=1}^{\lfloor n\theta' \rfloor} \big(E_k-1 \big)<-\epsilon+\frac1n -\frac{W_n}n,C_n\right]
\le \mathbb P\left[\frac{1}{n} \sum_{k=1}^{\lfloor n\theta' \rfloor} \big(E_k-1 \big)<-\frac23\epsilon+\frac1n\right]\to 0,
\] 
by the same argument.

\underline{Proof of $\frac{\PoisCN_{\PoisTau_n}}{n} \longrightarrow\inp{\mathsf{w}(\theta)}{\mathsf{1}} \ind{\mathsf{Surv^{MGW}}}$  in probability as $n\to\infty$.}\\
Since $\frac{\PoisN^i_{ns}}{n}\to \gamma_i (1-e^{-\beta_i s})=w_i(s)$ in probability, the law of large numbers together with a random time change (see e.g.\ \cite[p.\ 151]{Bill99}) and the first part of the proof yields
\begin{align*}
\frac{\PoisCN_{\PoisTau_n}}{n}&=\sum_{i=1}^J\underbrace{\dfrac{\PoisN^i_{\PoisTau_n}}{n}}_{\to \gamma_i(1-e^{-\beta_i \theta})\ind{\mathsf{Surv^{MGW}}} }\\
&\longrightarrow  \ind{\mathsf{Surv^{MGW}}} \sum_{i=1}^J \gamma_i(1-e^{-\beta_i \theta})=\inp{\mathsf{w}(\theta)}{\mathsf{1}}\ind{\mathsf{Surv^{MGW}}},
\end{align*}
in probability, as $n\to\infty$.  So this part of the proof is complete.

\underline{Proof of $\frac{\cN_{\tau_n}}{n} \longrightarrow\inp{\mathsf{w}(\theta)}{\mathsf{1}} \ind{\mathsf{Surv^{MGW}}}$  in probability as $n\to\infty$.}\\
Since $\tilde{\cN}_t=\cN_{\cP_t}$ for $t\in\mathds{N}$ and $\tau_n=\cP_{\PoisTau_n}$, we have $\tilde{\cN}_{\PoisTau_n}=\cN_{\tau_n}$ and the claim follows from the previous one.
\end{proof}

\section{Central Limit Theorems}\label{sec:clt}

In this section we prove central limit theorems for the duration $\tau_n$ of the process and for the total infected vertices at this time. We begin with a preliminary result.

\begin{theorem}\label{prop:Poisson_joint}
%\textcolor{teal}
{Suppose that \textsf{Assumption 1} holds}. For $t\in [0,\infty)$, and $n\in \N$ denote by 
\begin{align*}
Z^{(i)}_{nt}&=\dfrac{1}{\sqrt{n}}\left(
\PoisN_{nt}^i-n\gamma_i \cdot \left(1-e^{-\beta_i t}\right) \right)\\
Z^{\cP}_{nt}&=\dfrac{1}{\sqrt{n}}\left(\mathcal{P}_{nt}-nt\right).
\end{align*}
Then it holds
$$ \left\{\left(Z^{(1)}_{ns}, \dots, Z^{(J)}_{ns}, Z^{\cP}_{ns})\right); s\in[0,\infty)\right\}  \rightarrow \left\{\left(\sqrt{\gamma_1}\mathsf{X}_s^{(1)}, \dots, \sqrt{\gamma_J} \mathsf{X}_s^{(J)}, \mathsf{B}^{\mathcal{P}}
%_s
(s)\right), s\in[0,\infty)\right\},$$
as $n\to\infty$ in law in the Skorokhod space $\mathcal{D}([0,\infty), \mathbb{R}^{J+1})$ endowed with the standard $\mathbf{J}_1$ topology, where $(\sX^{(1)}, \dots, \sX^{(J)},\sB^{\cP})$ is a centered Gaussian process with the first $J$ components being independent and with covariances given by
\begin{align*}
\Cov[\sX^{(i)}_s,\sX^{(i)}_t] &=(1-e^{-\beta_is})e^{-\beta_i t} && s,t\in [0,\infty), \ s\le t, \\
\Cov[\sX^{(i)}_s,\sB^{\cP}(t)
%_t
] &= \min\{s,t\} \cdot\gamma_i e^{-s\beta_i} &&s,t\in[0,\infty), \\
\Cov[\sB^{\cP}(s),\sB^{\cP}(t)]
%\Cov[\sB^{\cP}_s,\sB^{\cP}_t] 
&= \min\{s,t\} &&s,t\in[0,\infty).
\end{align*}
\end{theorem}
\begin{proof}
The limit of $\left\{ Z_{ns}^{(i)}; s \in[0,\infty)\right\}$ is given in Proposition \ref{prop:clt-n_i}. We next investigate the limit of $\left\{ Z_{ns}^{\cP}; s \in[0,\infty)\right\}$. Since $\cP_t$ is a Poisson process with rate $1$, there are independent random variables $E_j \sim Exp(1), \, j\in\mathds{N}$ with 
$$\left\{\cP_t; t\in[0,\infty)\right\} = \left\{\sum_{k=1}^\infty \mathds{1}_{\{\sum_{j=1}^k E_j \le t\}};\  t\in[0,\infty)\right\}.$$ 
By using Donsker's invariance principle we first have that
\[ \left\{ \frac{1}{\sqrt{n}} \sum_{j=1}^{\lfloor ns \rfloor} (E_j-1); s\in[0,\infty)\right\}\stackrel{\law}{\longrightarrow} \left\{ \sB(s);\ s\in[0,\infty)\right\},\quad \text{as } n\to\infty\]
in the Skorokhod topology.
We then obtain
$$\left\{Z_{ns}^{\cP}; s\in[0,\infty)\right\} \stackrel{\law}{\longrightarrow} \left\{\sB(s);\  s\in[0,\infty)\right\}$$
in view of \cite[Theorem 14.6]{Bill99}, and we set $\sB^{\cP}(s):=\sB(s)$. 
Now let us turn to $\Cov[\PoisN_{ns}^i, \cP_{nt}]$. We assume w.l.o.g.\ that the vertices $v_1, \dots, v_{n_i}$ are of type $i$, while the vertices $v_{n_i+1},\dots, v_n$ are not of type $i$.  The time at which vertex $v_k$, $k=1,\dots,n$ receives the virus for the $l$-th time can be modeled by 
$t_{k,l}=\sum_{j=1}^l E_{k,j}$, where again $E_{k,j}$, $k=1,\dots, n$, $j\in\mathds{N}$, are independent exponentially distributed random variables with  
rate $p_{\type(v_k)}$. Then
$$ \left\{\PoisN_t^i; t\in[0,\infty)\right\}\stackrel{\law}{\sim}\left\{\sum_{k=1}^{n_i} \ind{t_{k,1}\le t};\  t\in[0,\infty)\right\} $$
and
$$ \left\{\cP_t; t\in[0,\infty)\right\}\stackrel{\law}{\sim} \left\{\sum_{k=1}^n \sum_{l=1}^\infty \ind{t_{k,l}\le t};\  t\in[0,\infty)\right\}.$$
Observe that 
$\cP_t(k):= \sum_{l=1}^\infty \ind{t_{k,l}\le t}\sim 
\mathsf{Poi}(tp_{\type(v_k)})$  and $\ind{t_{k,1}\le t}= \ind{\cP_t(k)\ge 1}$. For $s\le t$:
\begin{align*}
\frac{1}{n}& \Cov[\PoisN_{ns}^i,\cP_{nt}] = \frac{1}{n} \sum_{k=1}^{n_i}\left(\Cov\left[ \ind{t_{k,1}\le ns},\cP_{ns}(k)\right] + \Cov\left[\ind{t_{k,1}\le ns}, \cP_{nt}(k)-\cP_{ns}(k)\right]\right)  \\
&= \frac{1}{n}\sum_{k=1}^{n_i} \left(\E[ \cP_{ns}(k)] - \E[\cP_{ns}(k)] \cdot \E[\ind{\cP_{ns}(k)\ge 1}]\right) 
= \frac{1}{n} \sum_{k=1}^{n_i} s\beta_{i}\cdot e^{-s\beta_{i}}=s\cdot\gamma_ie^{-s\beta_i}.
\end{align*}
If $s\ge t$ we get
\begin{align*}
\frac{1}{n} &\Cov[\PoisN_{ns}^i,\cP_{nt}] = \frac{1}{n} \sum_{k=1}^{n_i}\Cov\left[ \ind{t_{k,1}\le ns},\cP_{nt}(k)\right]\\
&= \frac{1}{n}\sum_{k=1}^{n_i} \left(\E[ \cP_{nt}(k)] - \E[\cP_{nt}(k)] \cdot \E[\ind{t_{k,1}\le ns}]\right) 
=\frac{1}{n} \sum_{k=1}^{n_i} t\beta_{i}\cdot e^{-s\beta_{i}}=t \gamma_ie^{-s\beta_i}.
\end{align*}
The covariances $\Cov[\sX^{(i)}_s,\sX^{(i)}_t]$ have been computed in Proposition \ref{prop:clt-n_i}. Hence the finite-dimensional convergence of the processes of $\{Z^{(i)}_{ns}; s\in[0,\infty)\}$ to $\{\sqrt{\gamma_i}\sX^{(i)}_s;\ s\in[0,\infty)\}$ for all $i\in\floor{J}$ and of $\{Z^{\cP}_{ns};s\in[0,\infty)\}$ to $\{\sB^{\cP}%_s
(s);\ s\in[0,\infty)\}$ imply the finite-dimensional convergence of $ \left\{\left(Z^{(1)}_{ns},\dots, Z^{(J)}_{ns}, Z^{\cP}_{ns})\right); s\in[0,\infty)\right\} $ to $ \left\{\left(\sqrt{\gamma_1}\mathsf{X}_s^{(1)}, \dots, \sqrt{\gamma_J}\mathsf{X}_s^{(J)}, \mathsf{B}^{\mathcal{P}}
%_s
(s)\right), s\in[0,\infty)\right\}$. In order to prove tightness of the involved processes, let $\epsilon>0$. Then there are sets $\mathcal{K}_1,\dots, \mathcal{K}_{J+1} \subseteq \mathcal{D}([0,\infty), \mathbb{R})$ which are compact w.r.t.\ the Skorokhod topology such that
\begin{align*}
\mathbb{P}\left[ (s\mapsto Z^{(i)}_{ns}) \in \mathcal{K}_i \text{ for all } n\in\mathds{N}\right]\ge 1-\epsilon/(2J)\\
\mathbb{P}\left[ (s\mapsto Z^{\cP}_{ns}) \in \mathcal{K}_{J+1} \text{ for all } n\in\mathds{N}\right] \ge 1-\epsilon/2.
\end{align*}
Hence
\[ \mathbb{P}\left[ (s\mapsto (Z^{(1)}_{ns},\dots, Z^{(J)}_{ns}, Z^{\cP}_{ns})) \in \mathcal{K}_1 \times \dots \times \mathcal{K}_{J+1} \text{ for all }n\in\mathds{N}\right] \ge 1-\epsilon \]
and this implies tightness, so the claim is proved.
\end{proof}

Let $r_i(n):=\sum_{l=1}^n\L^i_l$  for $i\in [J]$ and $n\in\mathds{N}$.
\begin{lemma}\label{l:asy_independence}
Denote for $i\in[J]$
\begin{equation*}
Z_{i}^L(nq)= n^{-1/2} \left(r_i(nq)-nq\E[\L^i]\right) \quad \text{for }\ q\geq 0
\end{equation*}
and
\[Z_{i}^N(ns) = n^{-1/2}\left(\PoisN_{ns}^i-n\gamma_i(1-e^{-\beta_i s}) \right) \quad \text{for } s\geq 0.\]
If
\begin{equation}\label{e:Z_converge}
 \begin{pmatrix} \left\{ Z_{i}^L(nq);\ q \geq 0\right\} \\ \left\{ Z_{i}^N(ns);\ s \geq 0\right\} \end{pmatrix} \stackrel{\law}{\longrightarrow}  \begin{pmatrix} \left\{ \mathsf{Y}_{i}^L(q);\  q\ge 0 \right\}\\ \left\{ \mathsf{Y}_{i}^N(s);\  s\ge 0 \right\} \end{pmatrix} , \quad \text{as }n \to\infty
\end{equation}
in the product topology of the Skorokhod $\mathbf{J}_1$ topology under $\mathbb{P}$ for some stochastic process $(\{ \mathsf{Y}_{i}^L(q);\  q\ge 0 \}, \{ \mathsf{Y}_{i}^N(s);\  s\ge 0 \})$, then the same holds also under $\mathbb{P}\left[ \cdot \mid \mathsf{Surv}^{\mathsf{MGW}}\right]$ and under $\mathbb{P}\left[ \cdot \mid \left(\mathsf{Surv}^{\mathsf{MGW}}\right)^{\mathsf{C}}\right]$.
\end{lemma}
 
\begin{proof}Recall first the survival event 
$\mathsf{Surv^{MGW}}=\left\{\sum_{i=1}^J\sZ^i_t > 0,\ \forall t\in \mathds{N}\right\}$,
and  $\sZ_{\mathsf{tot}}^{\mathsf{MGW}}=\sum_{t=0}^{\infty}\sum_{i=1}^J\sZ^i_t$ represents the total number of individuals of the MGW $(\sZ_t)_{t\in \N}$.
From \eqref{eq:rv-li}  we know that, for the event $\{\sZ_{\mathsf{tot}}^{\mathsf{MGW}}\leq A\}$ for $A\geq 0$, the variables $\L_l^i$ in the definitions of $r_i(t)$ coincide with $\bar{\L}_l^i$ for $l>A$, so that conditionally on $\{\sZ_{\mathsf{tot}}^{\mathsf{MGW}}\leq A\}$ for $i=1,\ldots,J$ the processes 
$$\sum_{l=A}^{\lfloor nq \rfloor}\bar{\L}^i_l=r_i(nq)-\sum_{l=1}^A\L^i_l$$
are independent of the $\mathsf{MGW}$, and also independent among them. Moreover, for each vertex $j=1,\dots, n_i$ of type $i$ let $Y^{(j)}_l$, $l\in\mathds{N}$, denote the arrival times of the infection attempts of vertex $i$. Then 
\[ \PoisN_{t}^i= \sum_{j=1}^{n_i} \max_{l\in\mathds{N}} \mathds{1}_{\left\{Y_l^{(j)} \le t\right\}}. \]
Now the process
\[ \PoisN_{t}^{A,i} = \sum_{j=1}^{n_i} \max_{l\in\mathds{N}} \mathds{1}_{\left\{A<Y_l^{(j)} \le t\right\}} \]
is independent of $\{\sZ_{\mathsf{tot}}^{\mathsf{MGW}}\leq A\}$ and 
\[ \sup_{t\in[0,\infty)} \left|\PoisN_{t}^{A,i} - \PoisN_{t}^i\right| \le \sum_{j=1}^{n_i} \max_{l\in\mathds{N}} \mathds{1}_{\{Y_l^{(j)} \le A\}} \le \sum_{j=1}^{n_i} \sum_{l=1}^\infty \mathds{1}_{\{Y_l^{(j)} \le A\}} \]
and the right-hand side is a Poisson distributed random variable whose parameter does not depend on $n$. 
We have then that the event $\{\sZ_{\mathsf{tot}}^{\mathsf{MGW}}\leq A\}$ and the process\\
 $\left((Z^{L}_i(nq))_{q\geq 0}, (Z^{N}_i(ns))_{s\geq 0}\right)$ are asymptotically independent, thus also $\left(\mathsf{Surv^{MGW}}\right)^{\complement}$ and $\left((Z^{L}_i(nq))_{q\geq 0}, (Z^{N}_i(ns))_{s\geq 0}\right)$ are asymptotically independent. Therefore 
\[ \begin{pmatrix} \left\{n^{-1/2}\left(r_i(nq)-nq\E[\L^i]\right); {q\geq 0} \right\}   \\  \left\{ n^{-1/2}\left(\PoisN_{ns}^i-n\gamma_i(1-e^{-\beta_i s}) \right); s\geq 0 \right\}  \end{pmatrix} \]
 has the same limit in law under $\Pb$, under $\mathbb{P}\left[ \cdot \mid \left(\mathsf{Surv}^{\mathsf{MGW}}\right)^{\complement}\right]$ and hence also under $\mathbb{P}\left[ \cdot \mid \mathsf{Surv}^{\mathsf{MGW}}\right]$.
\end{proof}

If $\sigma_i^2:=\Var \left[\mathcal{L}^i\right]<\infty$ for all $i=1,\dots,J$, we define the following constants: 
\begin{equation}\label{eq:var-tau-til}
\sigma_{\PoisTau}^2=\dfrac{\sum_{i=1}^J\sigma_i^2\gamma_i(1-e^{-\beta_i \theta})+\sum_{i=1}^J \E[\L^i]^2\sigma_{\cN,i}^2+\theta-2\theta \sum_{i=1}^J\gamma_i^{3/2} e^{-\beta_i\theta}\E[\L^i]}{(1-\sum_{i=1}^J \beta_i\gamma_i \mathds{E}[\L^i]e^{-\beta_i \theta})^2}
\end{equation}
with $\sigma_{\cN,i}^2=\gamma_i(1-e^{-\beta_i\theta})e^{-\beta_i \theta} $,
\begin{align*}
\sigma_{\tau}^2&= \dfrac{1}{(1-\sum_{i=1}^J \beta_i\gamma_i \mathbb{E}[\L^i]e^{-\beta_i \theta})^2}\Bigg[\sum_{i=1}^J\sigma_i^2\gamma_i(1-e^{-\beta_i \theta})+\sum_{i=1}^J \E[\L^i]^2\sigma_{\cN,i}^2\\
&+\theta\cdot \left(\sum_{i=1}^J \beta_i\gamma_i \mathbb{E}[\L^i]e^{-\beta_i \theta}\right)^2 - 2\theta \left(\sum_{i=1}^J\gamma_i^{3/2} e^{-\beta_i\theta}\E[\L^i]
\right)
 \left(\sum_{i=1}^J \beta_i\gamma_i \mathds{E}[\L^i]e^{-\beta_i \theta}\right)\Bigg]
\numberthis \label{eq:var-tau}
\end{align*}
and
\begin{equation}\label{eq:var-ntau}
\sigma_w^2=c_w^2\sum_{i=1}^J\sigma_i^2\gamma_i(1-e^{-\beta_i \theta})+\sum_{i=1}^J(1+c_w\E[\L^i])^2\sigma_{\cN,i}^2-2c_w\theta\Bigg(\sum_{i=1}^J(1+c_w\E[\L^i]) \gamma_i^{3/2}e^{-\beta_i\theta}\Bigg)+c_w^2\theta
\end{equation}
with 
\begin{equation} c_w:=\frac{
%\textcolor{blue}
{\sum_{i=1}^J \beta_i\gamma_i e^{-\beta_i \theta}}}{1-\sum_{i=1}^J \beta_i\gamma_i \mathbb{E}[\L^i]e^{-\beta_i \theta}}, \label{eq:cw_def} \end{equation}

\begin{theorem}\label{thm:clt}\textbf{Central Limit Theorems for $\tau_n,\tilde\tau_n,$ and $\tilde{\cN}_{\PoisTau_n}$.}\\
Suppose 
%\textcolor{teal}
{\textsf{Assumption 1} and \textsf{Assumption 2} hold}, and $\sigma_i^2:=\Var \left[\mathcal{L}^i\right]<\infty$ for all $i=1,\dots,J$, i.e. all entries of the offspring distribution matrix $L$ of the $\mathsf{MGW}$ have finite variances.  Then we have
\begin{align*}
& n^{-1/2}(\PoisTau_n - n\theta) \stackrel{\law}{\longrightarrow} \mathcal{N}(0, \sigma_{\PoisTau}^2)\quad &&\text{as}\quad n\to\infty\\
&n^{-1/2}(\tau_n - n\theta) \stackrel{\law}{\longrightarrow} \mathcal{N}(0,\sigma_{\tau}^2),\quad &&\text{as}\quad n\to\infty\\
& n^{-1/2}(\tilde{\cN}_{\PoisTau_n}-nw) \stackrel{\law}{\longrightarrow} \mathcal{N}(0, \sigma_w^2), \quad && \text{as}\quad n\to\infty
\end{align*}
conditionally on $\mathsf{Surv}^{\mathsf{MGW}}$, i.e.\ under the probability measure $\mathbb{P}\left[ \cdot | \mathsf{Surv}^{\mathsf{MGW}}\right]$, where \\ $w:=\inp{\mathsf{w}(\theta)}{\mathsf{1}}=\sum_{i=1}^J \gamma_i(1-e^{-\beta_i \theta}) $
and $\sigma_{\PoisTau}^2,\sigma_{\tau}^2,\sigma_{w}^2$ are given by \eqref{eq:var-tau-til}, \eqref{eq:var-tau}, \eqref{eq:var-ntau} respectively. %for $\E[\mathcal{K}]=\sum_{i,j=1}^{n}\E[L^{(i,j)}]$
\end{theorem}

\begin{proof} 
The random variables $\L^1,\ldots,\L^J$ are independent by construction. By Proposition \ref{prop:var-li} $(\L^i_l)_{l\in \N}$ is an i.i.d.\ copy of $\L^i$, and also independent of $\sN^i_t$, so the random variables $\mathcal{R}^i_t=\sum_{l=1}^{\sN_t^i}\L_l^i,$ for $i=1,\ldots,J$ are independent as well, and we can apply for each of the processes $\left\{\sum_{l=1}^{nq}\L^i_l;\ q\geq 0\right\}$ the Donsker's invariance principle, and then a random time change.

For simplicity of notation, for $i\in [J]$ and $n\in\mathds{N}$,
let us write $r_i(n):=\sum_{l=1}^n\L^i_l$ so that $\mathcal{R}_t^i=r_i(N_t^i)$.
By Donsker's invariance principle we have
$$\left\{\frac{1}{\sqrt{n}}\left(r_i(nq)-nq\E[\L^i]\right);\ q\geq 0\right\}\stackrel{\law}{\longrightarrow}\left\{
\sigma_i \tilde{\sB}^i({q})
; \ q\geq 0\right\},\quad \text{as  } n\to\infty$$
in the Skorokhod topology, where
$\tilde{\sB}^i$ %_{q}$ 
is a Brownian motion, and 
$\sigma_i^2=\Var[\L^i]$. Let $\sX_s^{(i)}$ be the centered Gaussian process from Proposition \ref{prop:clt-n_i} which arises as scaling limit of $\PoisN_{ns}^i$. Since the random variables $\L_l^i$ are independent of $\sN^i_{ns}$ we also have the following joint convergence:
$$ \begin{pmatrix} \big\{ n^{-1/2} (r_i(nq)-nq\mathbb{E} [\L^i]); q\ge 0 \big\} \\  \big\{n^{-1/2}\left(\PoisN_{ns}^i-n\gamma_i(1-e^{-\beta_i s}) \right); s\geq 0\big\} 
\end{pmatrix}\stackrel{\law}{\longrightarrow}
\begin{pmatrix} \big\{
\sigma_i
\tilde{\sB}^i(q); %_{q}; 
\ q\geq 0\big\}\\ \big\{\sqrt{\gamma_i}\mathsf{X}^{(i)}_s;\ s\geq 0\big\} \end{pmatrix} $$
in the product topology, as $n\to \infty$, where both components are equipped with the Skorokhod topology. By Lemma \ref{l:asy_independence} this convergence holds under $\mathbb{P}\left[\cdot| \mathsf{Surv^{MGW}}\right]$ as well. Using that, by definition we have the sequence of equalities  
$$r_i(\PoisN^i_{ns})=r_i(\sN^i_{\mathcal{P}_{ns}})=\mathcal{R}^i_{\mathcal{P}_{ns}}=\sum_{l=1}^{\sN^i_{\mathcal{P}_{ns}}}\L_l^i=\sum_{l=1}^{\PoisN^i_{ns}}\L_l^i=\tilde{\mathcal{R}}^i_{ns}$$
and doing a random change of time (\cite[Lemma p.\ 151]{Bill99}) in $r_i(nq)$, we obtain
$$\left\{\frac{1}{\sqrt{n}}\left(\tilde{\mathcal{R}}^i_{ns}-\PoisN^i_{ns}\E[\L^i]\right); \ s\geq 0\right\}\stackrel{\law}{\longrightarrow}\left\{
\sigma_i
\tilde{\sB}^i ({w_i(s)}); \ s\geq 0
\right\}, $$
where we recall that $w_i(s)=\gamma_i(1-e^{-\beta_i s})$.
Hence 
\begin{align*}
\left\{\frac{1}{\sqrt{n}}\left(\tilde{\mathcal{R}}^i_{ns}-n\gamma_i\E[\L^i](1-e^{-\beta_i s}) \right); \ s\geq 0 \right\} \stackrel{\law}{\longrightarrow}
\left\{\sigma_i
\tilde{\sB}^i({w_i(s)})
+\E[\L^i]\sqrt{\gamma_i}\sX^{(i)}_s; \ s\geq 0\right\},
\end{align*}
where $\{\sX_s^{(i)}; s\ge 0\}$ is independent of $\{\tilde B_t^i; t \ge 0\}$. 
Since the processes $r_i(t)$, $i=1,\ldots, J$, (and thus also the processes $\tilde{\mathcal{R}}^i_t$) are independent, we also have convergence in law for the sum $\tilde{\mathcal{R}}_{ns}:=\sum_{i=1}^J \tilde{\mathcal{R}}^i_{ns}$. Therefore the process 
$$\left\{\frac{1}{\sqrt{n}}\left(\tilde{\mathcal{R}}_{ns}-n\sum_{i=1}^J\E[\L^i]\gamma_i(1-e^{-\beta_i s}) \right); \ s\geq 0 \right\} $$
converges in law to 
$$\left\{\sum_{i=1}^J\left(\sigma_i\tilde{\sB}^i({w_i(s)})+
\E[\L^i]\sqrt{\gamma_i}\sX^{(i)}_s\right); \ s\geq 0\right\},$$
where the processes $\{ \sX_s^{(i)}; s \ge 0\}$, $\{\tilde B_t^i; t\ge 0\}$, $i\in [J],$ are independent of each other. 
Now we know that at the end of the process (which occurs at time $\tau_n$) we have
$$\tilde{\mathcal{R}}_{\tilde\tau_n}=\mathcal{R}_{\tau_n}=\tau_n=\cP_{\tilde \tau_n}.$$ 
So $\tilde\tau_n$ is the point at which the process $\{\tilde{\mathcal{R}}_t - \mathcal{P}_t \mid t\in [0,\infty) \}$ becomes zero. Due to the independence mentioned above 
together with the scaling limit for $Z_{ns}^{\cP}$ from Theorem \ref{prop:Poisson_joint}, we get the convergence in law
\begin{align*}
&\left\{\frac{1}{\sqrt{n}}\left(\tilde{\mathcal{R}}_{ns}-\cP_{ns}+ns-n\sum_{i=1}^J\E[\L^i]\gamma_i(1-e^{-\beta_i s})\right); \ s\geq 0 \right\} 
\stackrel{\law}{\longrightarrow} \\
&
\left\{
\sum_{i=1}^J\sigma_i\tilde{\sB}^i({w_i(s)})
+\sum_{i=1}^J\E[\L^i]\sqrt{\gamma_i}\sX_s^{(i)}-\sB^{\cP}(s);
\ s\geq 0\right\},
\end{align*}
where $\sB^{\cP}$ is a Brownian motion
 independent of the family $\{\tilde{\sB}^i\}_{1\le i\le J}$
 and correlated to $\{\sX^{(i)}\}_{1\le i \le J}$ as in Theorem \ref{prop:Poisson_joint}.  
By Theorem \ref{t:LLN}, both $\tau_n/n$ and $\tilde\tau_n/n$ converge in probability to $\theta\ind{\mathsf{Surv^{MGW}}}$, with $ns=\tilde\tau_n$
and we deduce that
$$n^{1/2}\cdot \underbrace{\left(\frac{\tilde\tau_n}{n}-\sum_{i=1}^J\E[\L^i]\gamma_i(1-e^{-\beta_i \tilde\tau_n/n})\right)}_{:= -f(\tilde\tau_n/n)} \stackrel{\law}{\longrightarrow} 
\sum_{i=1}^J\sigma_i\tilde{\sB}^i({w_i(\theta)})
+\sum_{i=1}^J\E[\L^i]\sqrt{\gamma_i}\sX^{(i)}_{\theta}-\sB^{\cP}(\theta)
$$
 as $n\to\infty$, 
with $f(s)$ as defined in \eqref{eq:ftheta}.
Since $f(\theta)=0$, a Taylor series expansion around $\theta$ gives 
$$ \frac{\tilde\tau_n}{n} - \theta = \frac{-f(\frac{\tilde\tau_n}{n})}{-f'(\theta)- \frac{1}{2}(\frac{\tilde\tau_n}{n}-\theta) f''(A)} $$
for some $A\in [\min\{\tilde\tau_n/n, \theta\}, \max\{\tilde\tau_n/n,\theta\}]$. 
By the law of large numbers, Theorem \ref{t:LLN}, we have $f'(\theta)+ \frac{1}{2}(\frac{\tilde\tau_n}{n}-\theta) f''(A)\longrightarrow f'(\theta)$ as $n\to\infty$ in probability. 
Thus
$$n^{-1/2}(\tilde\tau_n - n\theta) \stackrel{\law}{\longrightarrow} \frac{1}{-f'(\theta)} \cdot\left( %\tilde{\sB}\left(\sum_{i=1}^J\sigma_i^2\gamma_i(1-e^{-\beta_i \theta})\right)
\sum_{i=1}^J\sigma_i\tilde{\sB}^i({w_i(\theta)})
+\sum_{i=1}^J\E[\L^i]\sqrt{\gamma_i}\sX_{\theta}^{(i)}-\sB^{\cP}(\theta)
\right)\sim \mathcal{N}(0,\sigma_{\tilde\tau}^2),$$
with $\sigma_{\tilde\tau}^2$ defined as in equation \eqref{eq:var-tau-til} %and $\sigma_{\cN,i}^2=\gamma_i(1-e^{-\beta_i\theta})e^{-\beta_i \theta} $ 
and this completes the first of the three claims (being $-f'(\theta)=1-\sum_{i=1}^J\beta_i\gamma_i\E[\L^i]e^{-\beta_i\theta}$).

The second claim follows easily from the first one by using that $\tau_n=\cP_{\tilde\tau_n}$ together with 
Theorem \ref{prop:Poisson_joint} we also get
\begin{align*} 
&n^{-1/2}(\tau_n-n\theta)=n^{-1/2}(\cP_{\tilde\tau_n}-\tilde\tau_n+\tilde\tau_n-n\theta)\\
& \stackrel{\law}{\longrightarrow}\sB^{\cP}(\theta)
  +  \frac{1}{-f'(\theta) %1-\sum_{i=1}^J \beta_i\gamma_i \mathbb{E}[\L^i]e^{-\beta_i \theta}
} \cdot\left( 
\sum_{i=1}^J \sigma_i \tilde{\sB}^i({w_i(\theta)})
+\sum_{i=1}^J\E[\L^i]\sqrt{\gamma_i}\sX_{\theta}^{(i)}-\sB^{\cP}(\theta)
\right)\\
& \sim \mathcal{N}(0,\sigma_{\tau}^2)
\end{align*}
with $\sigma_{\tau}^2$ given by equation \eqref{eq:var-tau}.

Finally, for the third claim we proceed as follows:
\[ n^{-1/2}\left(\PoisCN_{\PoisTau_n}-n\inp{\mathsf{w}(\theta)}{\mathsf{1}}\right)=
n^{-1/2}\left(\PoisCN_{\PoisTau_n}-n\inp{\mathsf{w}(\PoisTau_n/n)}{\mathsf{1}}+n\inp{\mathsf{w}(\PoisTau_n/n)}{\mathsf{1}}-n\inp{\mathsf{w}(\theta)}{\mathsf{1}}\right).\]
Now 
Theorem \ref{prop:Poisson_joint} together with 
\[n \langle \mathsf{w}(\PoisTau_n/n)-\mathsf{w}(\theta), \mathsf{1} \rangle = 
\sum_{i=1}^J \beta_i\gamma_i e^{-\beta_i\theta}(\PoisTau_n-n\theta)+ O\left(n\left(\frac{\PoisTau_n}n-\theta\right)^2\right) \]
and 
\[ n^{1/2}\left(\frac{\PoisTau_n}{n}-\theta\right)^2 \longrightarrow 0\]
in probability as $n\to\infty$ yields
\begin{align*}
&n^{-1/2}\left(\PoisCN_{\PoisTau_n}-n\inp{\mathsf{w}(\theta)}{\mathsf{1}}\right)\\
&\stackrel{\law}{\longrightarrow} \sum_{i=1}^J \sqrt{\gamma_i} \sX^{(i)}_{\theta} 
%\textcolor{blue}
{+} 
%\textcolor{teal}
{c_w}\cdot \left( 
\sum_{i=1}^J \sigma_i\tilde{\sB}^i({w_i(\theta)})
+\sum_{i=1}^J\E[\L^i]\sqrt{\gamma_i}\sX^{(i)}_{\theta}-\sB^{\cP}(\theta)
\right)\\
&\sim \mathcal{N}(0,\sigma^2_w),
\end{align*}
where 
%\textcolor{teal}
{$c_w$ and $\sigma_w^2$ are defined as in \eqref{eq:cw_def} and \eqref{eq:var-ntau}.
This completes the whole proof.}
\end{proof}

\begin{corollary}
Since $\tilde{\cN}_t=\cN_{\cP_t}$ for $t\in\mathds{N}$ and $\tau_n=\cP_{\PoisTau_n}$, we have $\tilde{\cN}_{\PoisTau_n}=\cN_{\tau_n}$ and thus 
$$n^{-1/2}(\cN_{\tau_n}-nw) \stackrel{\law}{\longrightarrow} \mathcal{N}(0, \sigma_{w}^2), \quad  \text{as}\quad n\to\infty$$
with $w$ and $\sigma_{w}^2$ as in Theorem \ref{thm:clt}.
\end{corollary}
We would like to emphasize that Theorem \ref{t:LLN} and Theorem \ref{thm:clt} are generalizations of \cite[Theorem 2.2]{comets14} and \cite[Theorem 2.3]{comets14} respectively, with the main difference being in the fact that the coupon collector's process and the branching process are not independent, and several difficulties arise in this case. We also prove continuous-time counterparts of these results.

\section{Inhomogeneous versus homogeneous population}\label{sec:final}

In this section we want to compare what happens in the uniform and non-uniform case. As we have seen in Theorem \ref{t:LLN}, the results are apparently undistinguishable. On the event $\mathsf{Surv}^{\mathsf{MGW}}$, we have
\[
\frac{\tau_n}n \to \theta\quad \text{and} \quad  \frac{\cN_{\tau_n}}n \to p,\quad \text{ in probability, as  } n\to\infty.
\]
The difference lies in the computation of $\theta$ and $p$.
\paragraph*{Homogeneous spread capacity.} 
Let us first deal with the case where all individuals have identically distributed spread capacities, let this law be $\mathcal L$.
In a uniform population, all individuals have the same susceptibility and $\theta$ is the solution of the equation
\[
\theta = \E[\mathcal L]\cdot\left(1-e^{-\theta}\right).
\]
Let $\theta_{un}$ be this solution: the corresponding fraction of infected individuals is then
\[
p=p_{un}:=1-e^{-\theta_{un}}=\theta_{un}/ \E[\mathcal L].
\]

In the non-uniform case,
$\theta$ is the solution of the equation
\[
\theta = {\E[\mathcal L]}\left( 1-\sum_{i=1}^J \gamma_ie^{-\beta_i\theta}\right).
\]
We call this solution $\theta_\beta$ and we have $p=p_\beta= 1-\sum \gamma_ie^{-\beta_i\theta_\beta}=\theta_\beta/\E[\mathcal L]$.  

The following result says that in the non-uniform case, the epidemics dies out earlier and spreads in a smaller portion of the population.

\begin{figure}[h]
\centering
\begin{tikzpicture}[xscale=1.5,yscale=3]
\draw [<->] (0,1.1) -- (0,0) -- (3.1,0);

\draw[dotted] (0,0) to (3,1);

\node[below] at (0,0) {0};
\node[below] at (2.82144,0) {$\theta_{un}$};
\draw[dotted] (2.82144,0) -- (2.82144,0.94048);
\draw[dotted] (2.82144,0) -- (2.82144,0.94048);
\draw[dotted] (0,0.94048) -- (2.82144,0.94048);

\draw[dotted] (0,.690133) -- (2.0704,0.690133);
\draw[dotted] (2.0704,0) -- (2.0704,0.690133);
\node[below] at (2.0704,0) {$\theta_\beta$};

\node[left] at (0,0.690133) {{$p_\beta$}};
\node[left] at (0,0.94048) {{$p_{un}$}};

\draw[dashed, thick, domain=0:3] plot (\x, {(1-exp(-\x)});
\node[right] at (3,0.95) {$y=g(x)$};
\draw[thick, domain=0:3] plot (\x, {1-(exp(-9*\x/4)+2*exp(-3*\x/8))/3});
\node[right] at (3,0.7831) {$y=f(x)$};
\end{tikzpicture}
\caption{Inhomogeneous (solid) vs homogeneous (dashed)}
\label{fig:inhom1}
\end{figure}
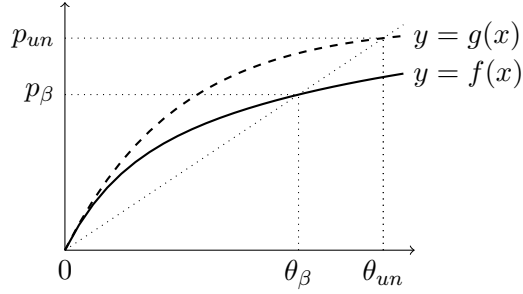

\begin{proposition}
For any fixed value of $\E[\mathcal L]>1$, and any choice of the parameters $\beta_i$,
\[
\theta_\beta\le\theta_{un} \quad \text{and} \qquad p_\beta\le p_{un},
\]
and equality holds if and only if $\beta_i=1$ for all $i\in[J]$.
\end{proposition}

\begin{proof}
If $\beta_i=1$ for all $i\in[J]$ the statement is trivial. Let us assume that there exists $\beta_i\neq 1$.
Define
\[
f(x)=1-\sum_{i=1} ^J\gamma_ie^{-\beta_ix} \qquad \text{and} \qquad g(x)=1-e^{-x}.
\]
Note that $(\theta_{un},p_{un})$ and $(\theta_\beta,p_\beta)$ are the intersections of $y=x/\E[\mathcal L]$ with $y=g(x)$ and $y=f(x)$, respectively.
It is enough to prove that $f(x)\le g(x)$ for all $x\ge0$ and equality holds only at 0 (see Figure \ref{fig:inhom1}).
The two functions are clearly equal at 0. Since the $\exp$-function is a 
strictly convex function and $\sum_{i=1}^J \gamma_i=1$, 
for $x>0$ we get
\[ f(x)=1-\sum_{i=1}^J \gamma_ie^{-\beta_ix} <   1-e^{-\sum_{i=1}^J \gamma_i\beta_ix}=1-e^{-x}=g(x), \]
where we have used $\sum_{i=1}^J \gamma_i\beta_i= \sum_{i=1}^J \alpha_i=1$.
This proves the statement.
\end{proof}

\paragraph*{Inhomogeneous spread capacity.} 
When types differ in the expected spread capacity $\E[\mathcal L^i]$, as well as in the susceptibility,  then 
 $\theta$ is the solution of
\[
\theta  = f(\theta):=\sum_{i=1}^J \E[\mathcal L^i]\gamma_i(1-e^{-\beta_i\theta}),
\]
and $p= 1- \sum_{i=1}^J\gamma_ie^{-\beta_i\theta}$.
We want to compare this case with the case where
all $\beta_i=1$ ( homogeneous susceptibility, inhomogeneous spread -- $\E[\mathcal L^i]$ are unaltered).
If we define
\[
g(x)=\sum_i  \E[\mathcal L^i]\gamma_i(1-e^{-x}),
\]
the parameters $\theta_\beta$
and $\theta_{un}$ are the $x$-coordinates  of the intersection of $y=x$ with $y=f(x)$ and $y=g(x)$, respectively. 
Depending on the choice of the parameters, the "homogeneous" case ($\beta_i=1$) may last longer or not, that is, both
$\theta_{un}>\theta_\beta$ and $\theta_{un}<\theta_\beta$ are possible scenarios.
Indeed there are examples where the inhomogeneous setting, compared to the homogeneous one, has:
\begin{enumerate}[a)]
\setlength\itemsep{0em}
\item a shorter epidemics involving a smaller percentage of cases;
\item
a longer epidemics involving a smaller percentage of cases;
\item a longer epidemics involving a larger percentage of cases.
\end{enumerate}
Here are examples of the three possible situations described before.
\begin{enumerate}[a)]
\setlength\itemsep{0em}
\item 
Take
$J=3$, $\gamma_i=1/3$ for all $i$, $ \E[\mathcal L^1]=1$, $ \E[\mathcal L^2]=2$, $ \E[\mathcal L^3]=4$, $\beta_1=2$, $\beta_2=0.5$, 
$\beta_3=0.5$, then $\theta_{un}\approx2.0255$, $\theta_\beta\approx 1.2139$, $p_{un}\approx 0.8681$, $p_\beta\approx0.6073$.
\item
Take $J=3$, $\gamma_i=1/3$ for all $i$, $ \E[\mathcal L^1]=1$, $ \E[\mathcal L^2]=2$, $ \E[\mathcal L^3]=4$, $\beta_1=0.6$, $\beta_2=1.2$, 
$\beta_3=1.2$:  $\theta_{un}\approx2.0255$, $\theta_\beta\approx 2.0703$, $p_{un}\approx 0.8681$, $p_\beta\approx0.8482$.
\item 
Let $J=2$, $\gamma_1=\gamma_2=1/2$, $ \E[\mathcal L^1]=0.2$,
$ \E[\mathcal L^2]=3.6$, $\beta_1=0.6$, $\beta_2=1.4$, $\theta_{un}\approx1.4578$, $\theta_\beta\approx 1.6964$, $p_{un}\approx 0.7672$, $p_\beta\approx0.7728$.
\end{enumerate}

Notice that in the inhomogeneous setting we cannot have a shorter epidemics involving a larger percentage of cases. Indeed, if $\theta_\beta < \theta_{un}$, then the convexity of the $\exp$-function implies
\[ p_\beta= 1- \sum_{i=1}^J \gamma_i e^{-\beta_i \theta_\beta} < 1- e^{-\sum_{i=1}^J \gamma_i \beta_i \theta_\beta} = 1-  e^{-\theta_\beta}<  1-  e^{-\theta_{un}} = 1- \sum_{i=1}^J \gamma_i e^{-\beta_i \theta_\beta} =p_{un}. \]

\begin{figure}[h]
\centering
\includegraphics[height=4cm]{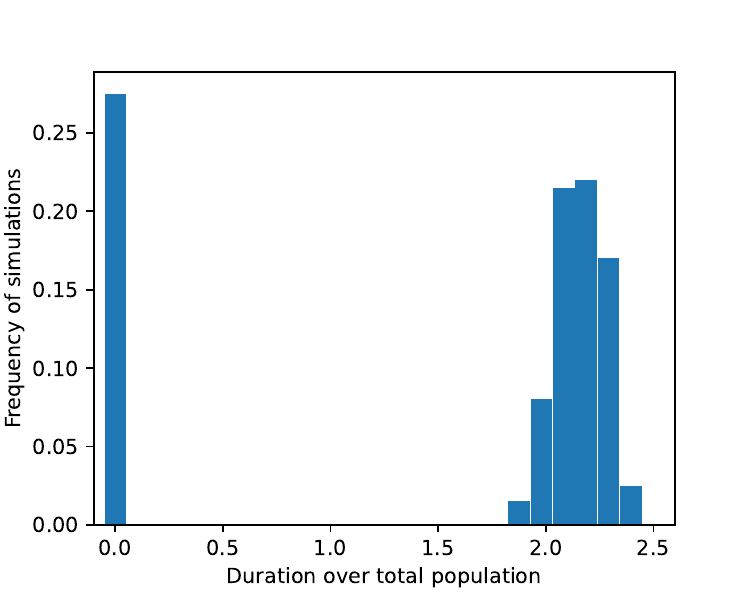}
\includegraphics[height=4cm]{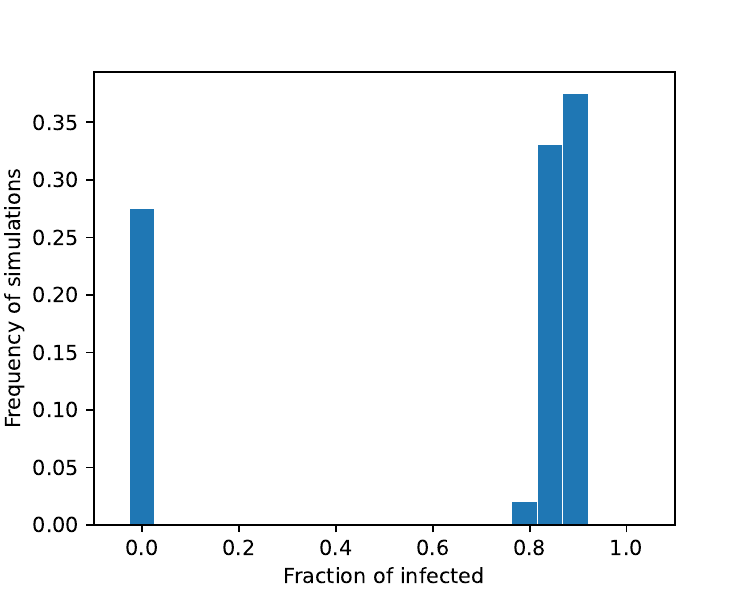}
\includegraphics[height=4cm]{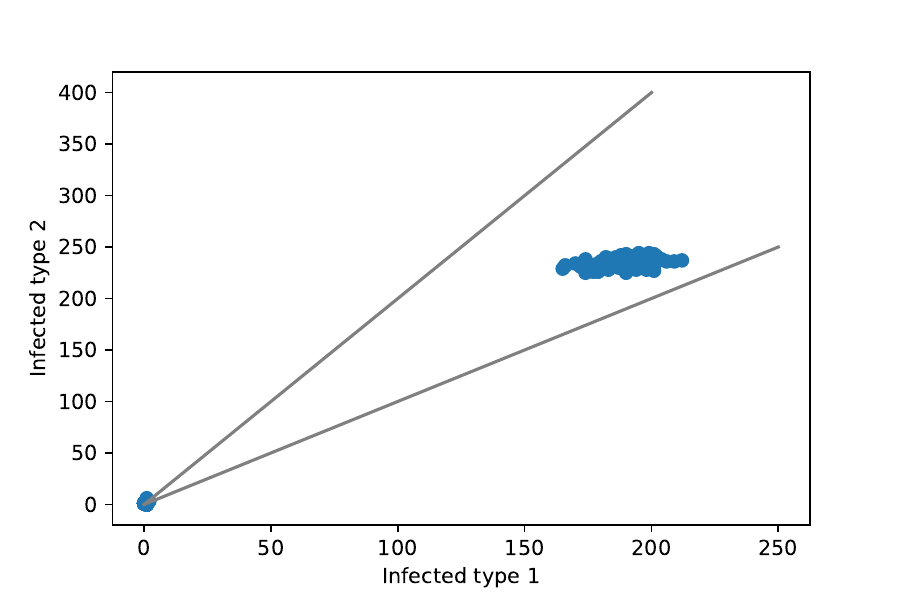}
\caption{Equal groups, second more susceptible}
\label{fig:simul1}
\end{figure}

One may also wonder what happens if we compare an epidemics which is inhomogeneous both in the spread and susceptibility, with
the case where the $\beta_i$ are unaltered and the spread capacities are homogeneous in the sense that all types have the same
 spread capacity $\mathcal L$, with $\E[\mathcal L]=\sum_i \gamma_i\E[\mathcal L^i]$. It is not difficult to find examples 
where the homogeneous case lasts longer (or less) than in the inhomogeneous case.

\paragraph*{Simulations.} 
Below we give the plots of some simulations, which agree with our theoretical results.
In all simulations we have $n=500$ and we run 200 simulations. Moreover, the spread capacity is homogeneous with expectation 2.5
(1/2 being the probability of 4 attempts and 1/4 each the probability of 0 or 2 attempts). The probability of extinction of the associated GW process is 0.2711.
There are two subpopulations which differ in susceptibility.
We have the histograms of the frequencies of $N_n(\tau_n)/n$ and of $\tau_n/n$. We also plot a scatterplot of the numbers of infected of the two types, with 
the two lines  $y= x \cdot\gamma_2/\gamma_1$ and $y = x\cdot \beta_2/\beta_1$. %Describing the behaviour of the total number of infected in the subpopulations is an open question.

In the first example depicted in Figure \ref{fig:simul1}, $\gamma_1=0.5$, $\gamma_2=0.5$, $\alpha_1=1/3$ and $\alpha_2=	2/3$.
The limiting values are $\theta = 2.12261$ and $p = 0.849044$.

\begin{figure}[h]
\centering
\includegraphics[height=4cm]{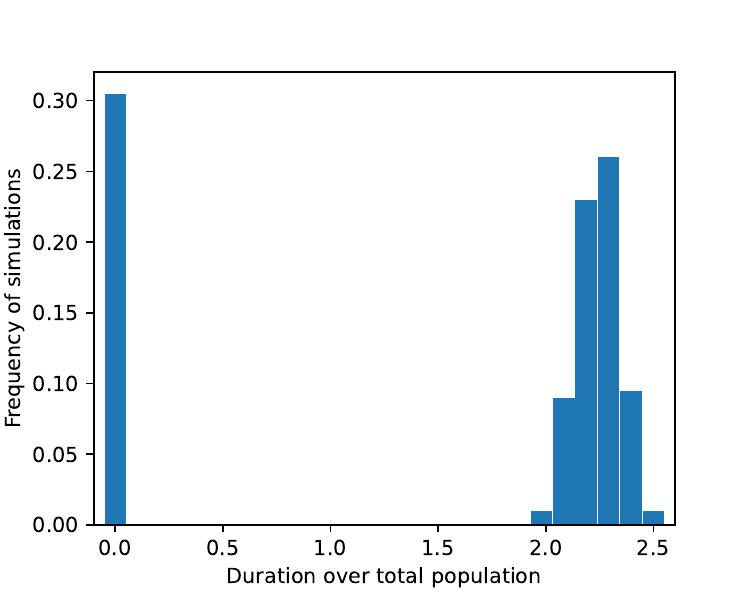}
\includegraphics[height=4cm]{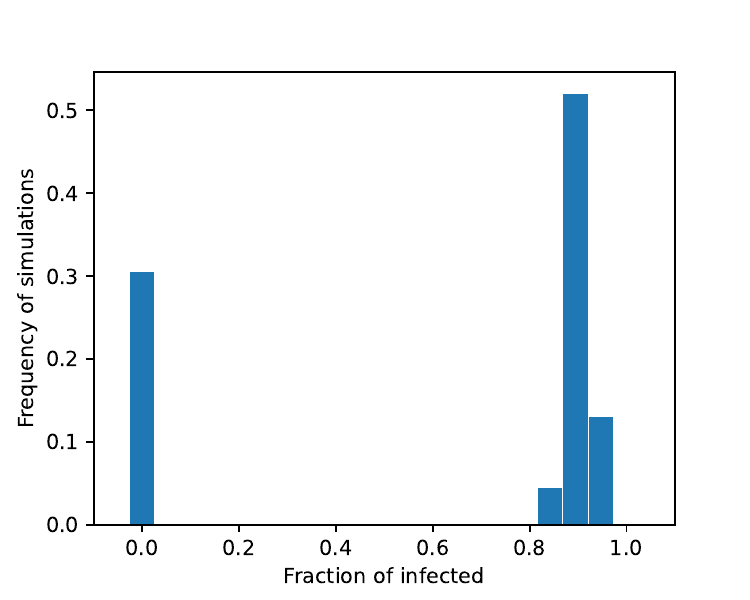}
\includegraphics[height=4cm]{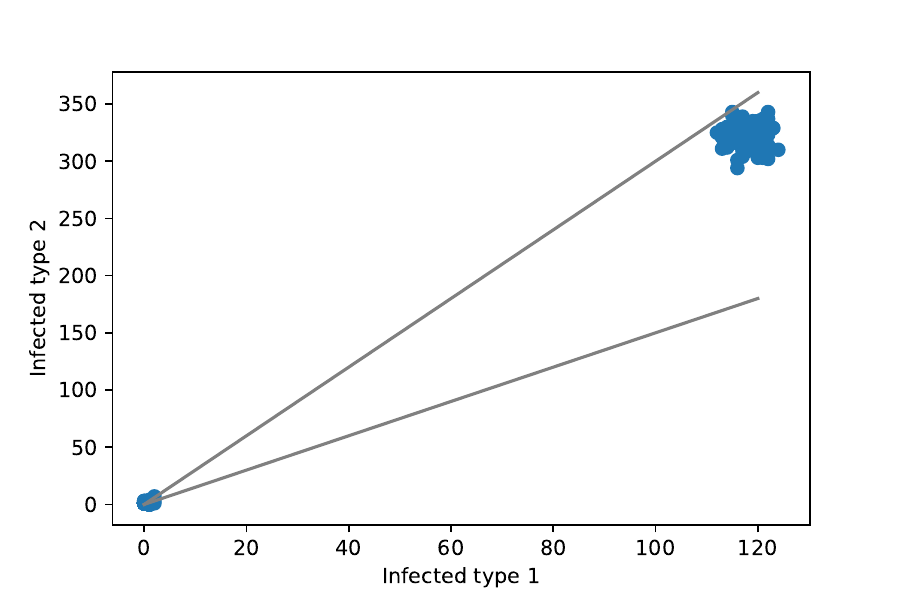}
\caption{Larger group is more likely}
\label{fig:simul2}
\end{figure}

In the second example depicted in Figure \ref{fig:simul2}, $\gamma_1=0.25$, $\gamma_2=0.75$, $\alpha_1=1/3$ and $\alpha_2=	2/3$.
The larger group is more likely to be chosen as a target, but single individuals in this group are less likely to be hit than individuals in the other group.
The limiting values are $\theta = 2.20202$ and $p = 0.880807$.

In the third example depicted in Figure \ref{fig:simul3}, $\gamma_1=0.75$, $\gamma_2=0.25$, $\alpha_1=1/3$ and $\alpha_2=	2/3$.
The larger group is less likely to be chosen as a target, and single individuals are less likely to be hit than individuals in the smaller group.
The limiting values are $\theta = 1.54723$ and $p = 0.618904$.

\begin{figure}[h]
\centering
\includegraphics[height=4cm]{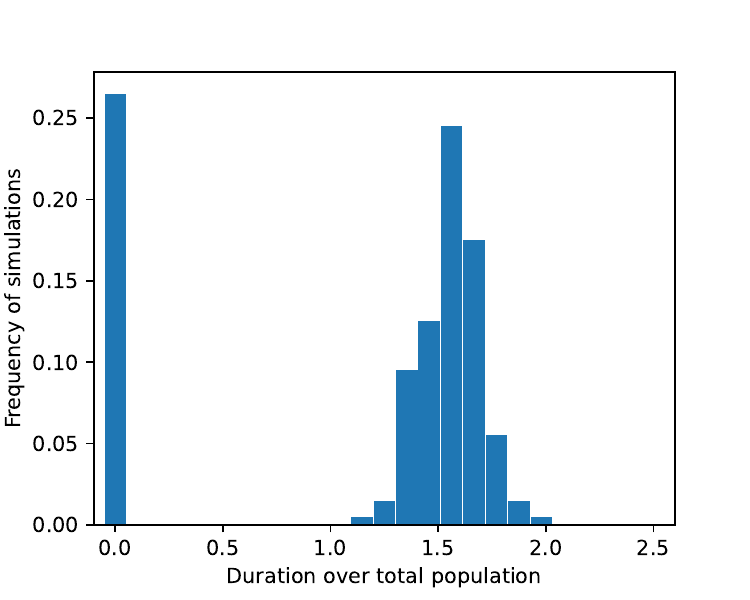}
\includegraphics[height=4cm]{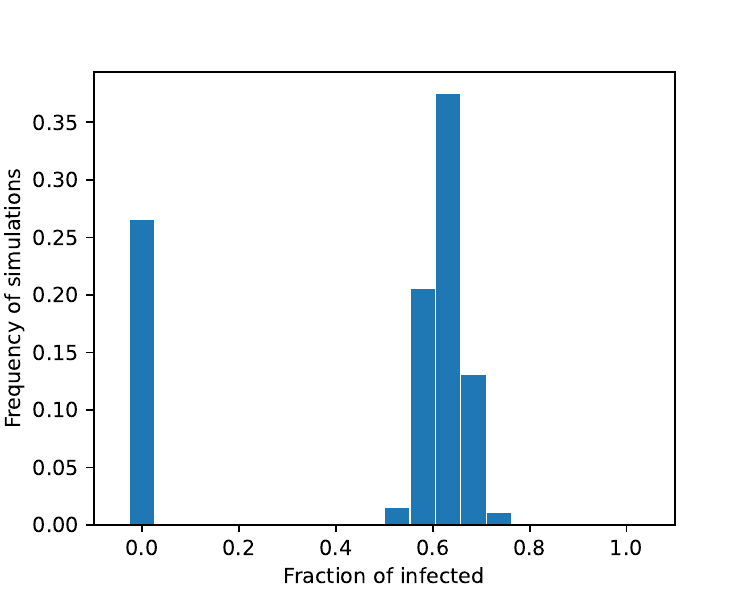}
\includegraphics[height=4cm]{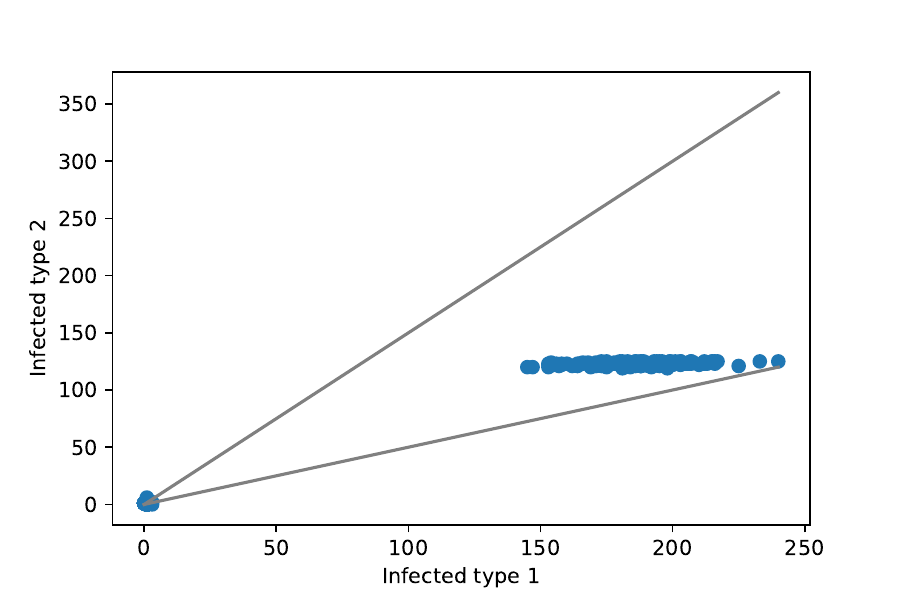}
\caption{Larger group is less likely}
\label{fig:simul3}
\end{figure}

\paragraph*{Further directions and questions.} We believe that our method can also be extended to the case when the underlying state space is a Erdös-Renyi random graph $G(n,p)$. Some of our results may carry over to the random setup, but the proofs will be very technical. Another possible extension of our model is to drop \textsf{Assumption 1}, and to consider the case where the groups of vertices $n_i$ of type $i$ and their weights $\alpha_i$ are functions of the total number of vertices $n$. Depending on the growth of these functions, we may expect different behavior and rescaling in the limit results.

\paragraph*{Acknowledgements.} We are grateful to the referee, whose comments and suggestions contributed to the improvement of the paper.
Daniela Bertacchi and Fabio Zucca acknowledge support by GNAMPA-INdAM; the research of
Ecaterina Sava-Huss is supported by the Austrian Science Fund (FWF): P 34129.

\bibliography{Virus-spread2.bib}{}

\bibliographystyle{abbrv}

\textsc{Daniela Bertacchi}, Department of Mathematics, University of Milano-Bicocca, Italy.\\
\texttt{daniela.bertacchi@unimib.it}

\textsc{Jürgen Kampf}, {University Hospital of Essen, Germany.\\
\texttt{juergen.kampf@uk-essen.de}}

\textsc{Ecaterina Sava-Huss}, Department of Mathematics, University of Innsbruck, Austria.
\texttt{Ecaterina.Sava-Huss@uibk.ac.at}

\textsc{Fabio Zucca},  Department of Mathematics, Polytechnic University of Milan, Italy.\\
\texttt{fabio.zucca@polimi.it}

\end{document}